\documentclass[10pt,reqno]{amsart}

\usepackage[a4paper, margin=1.5in]{geometry}
\usepackage {color}
\usepackage{amsmath,amsthm,amssymb}
\usepackage{epsfig}
\usepackage{graphicx}
\usepackage{amssymb}
\usepackage{latexsym}
\usepackage{pdfpages}

\usepackage{cancel}

\usepackage{subcaption}

\usepackage{tikz}
\usepackage{tkz-graph}

\usepackage{pgfplots}
\pgfplotsset{compat=1.15}
\usepackage{mathrsfs}
\usetikzlibrary{arrows}

\usepackage{pgf}

\usepackage{ulem} %para poder tachar: \sout{loquequierotachar}

\usepackage{ifpdf}
\newcommand{\res}{\!\!\mathop{\hbox{
                                \vrule height 7pt width .5pt depth 0pt
                                \vrule height .5pt width 6pt depth 0pt}}
                                \nolimits}

\def\z{{\bf z}}

\def\g{{\bf g}}
\def\h{{\bf h}}

\ifpdf %if using pdfLaTeX in PDF mode
  \usepackage[hidelinks]{hyperref}
\else %if using LaTeX or pdfLaTeX in DVI mode
\fi %%% fin del ifpdf

 \usepackage[usenames,dvipsnames]{pstricks}
 \usepackage{epsfig}
 \usepackage{pst-grad} % For gradients
 \usepackage{pst-plot} % For axes

\newtheorem{theorem}{Theorem}[section]
\newtheorem{lemma}[theorem]{Lemma}
\newtheorem{definition}[theorem]{Definition}
\newtheorem{proposition}[theorem]{Proposition}
\newtheorem{corollary}[theorem]{Corollary}
\newtheorem{remark}[theorem]{Remark}
\newtheorem{example}[theorem]{Example}

\newtheorem*{theorem*}{\it Theorem}

\def\vint_#1{\mathchoice%
          {\mathop{\kern 0.2em\vrule width 0.6em height 0.69678ex depth -0.58065ex
                  \kern -0.8em \intop}\nolimits_{\kern -0.4em#1}}%
          {\mathop{\kern 0.1em\vrule width 0.5em height 0.69678ex depth -0.60387ex
                  \kern -0.6em \intop}\nolimits_{#1}}%
          {\mathop{\kern 0.1em\vrule width 0.5em height 0.69678ex
              depth -0.60387ex
                  \kern -0.6em \intop}\nolimits_{#1}}%
          {\mathop{\kern 0.1em\vrule width 0.5em height 0.69678ex depth -0.60387ex
                  \kern -0.6em \intop}\nolimits_{#1}}}
\def\vintslides_#1{\mathchoice%
          {\mathop{\kern 0.1em\vrule width 0.5em height 0.697ex depth -0.581ex
                  \kern -0.6em \intop}\nolimits_{\kern -0.4em#1}}%
          {\mathop{\kern 0.1em\vrule width 0.3em height 0.697ex depth -0.604ex
                  \kern -0.4em \intop}\nolimits_{#1}}%
          {\mathop{\kern 0.1em\vrule width 0.3em height 0.697ex depth -0.604ex
                  \kern -0.4em \intop}\nolimits_{#1}}%
          {\mathop{\kern 0.1em\vrule width 0.3em height 0.697ex depth -0.604ex
                  \kern -0.4em \intop}\nolimits_{#1}}}

\def\R{\mathbb R}
\def\N{\mathbb N}
\def\g{\hbox{\bf g}}

\numberwithin{equation}{section}

\def\1{\raisebox{2pt}{\rm{$\chi$}}}

\def\g{{\bf g}}

\def\z{{\bf z}}

\definecolor{violet(ryb)}{rgb}{0.53, 0.0, 0.69}
\usepackage{mathtools}
\mathtoolsset{showonlyrefs}
\begin{document}

\title[Gelfand-Type problems  in RWS]
{Gelfand-Type problems in Random Walk Spaces }

\author[ J. M. Mazon, A. Molino and J. Toledo]{ J. M. Maz\'on, A. Molino and J. Toledo}

\address{J. M. Maz\'{o}n: Departamento de An\'{a}lisis Matem\`atico,
Universitat de Val\`encia, Dr. Moliner 50, 46100 Burjassot, Spain.
%\hfill\break\indent
 {\tt mazon@uv.es }}

\address{A.  Molino:  Departamento de Matem\'aticas, Facultad de Ciencias Experimentales,
Universidad de Almeria, Ctra. de Sacramento sn. 04120, La Ca\~nada de San Urbano, Almer\'ia, Spain.
%\hfill\break\indent
 {\tt amolino@ual.es }}

\address{J. Toledo: Departamento de An\'{a}lisis Matem\`atico,
Universitat de Val\`encia, Dr. Moliner 50, 46100 Burjassot, Spain.
%\hfill\break\indent
 {\tt toledojj@uv.es }}


\date{} %%\date{\today}

\begin{abstract}
This paper deals with Gelfand-type problems
 \begin{equation}\label{Gelfand10}
\qquad\qquad\left\{\begin{array}{ll} - \Delta_m u = \lambda f(u), \quad&\hbox{in} \ \Omega, \ \lambda >0, \\[10pt] u =0, \quad&\hbox{on} \ \partial_m\Omega,  \end{array} \right.
 \end{equation}
  in the framework of Random Walk Spaces, which includes as particular cases: Gelfand-type problems   posed on  locally finite weighted connected graphs and Gelfand-type problems driven by convolution  integrable kernels. Under the same assumption on the nonlinearity $f$ as in the local case, we show there exists an extremal parameter $\lambda^* \in (0, \infty)$ such that, for $0 \leq \lambda < \lambda^*$, problem \eqref{Gelfand10} admits a minimal bounded solution $u_\lambda$ and there are not solution for $\lambda > \lambda^*$. Moreover, assuming $f$ is convex,   we show that Problem \eqref{Gelfand10} admits a minimal bounded solution for $\lambda = \lambda^*$. We also show that $u_\lambda$ are stable, and, for $f$   strictly convex, we show that they are the unique stable solutions. We give simple examples that illustrate the many situations that can occur when solving Gelfand-type problems on weighted graphs.
\end{abstract}

\maketitle

% \today
%\\
%
%
% {\blue Alexis is NOT working. }
%\\
%
%
%{\red Mazon is NOT working. }
%\\
%
%
% {\violet Juli\'{a}n is NOT  working.  }

\section{Introduction}
\setcounter{equation}{0}

The classical Gelfand-type problem deals the existence and boundedness of positive solutions to the following semilinear elliptic equation
\begin{equation}\label{Gelfand0}
 \left\{\begin{array}{ll} - \Delta u = \lambda f(u) \quad&\hbox{in} \ \Omega, \\[10pt] u =0 \quad&\hbox{on} \ \partial\Omega,  \end{array} \right.
 \end{equation}
where $\lambda \geq 0$, $\Omega$ is an open bounded subset of $\R^n$ with Lipschitz–continuous boundary and  $f\in \mathcal{C}^1([0,\infty[)$ satisfying the following hypotheses
\begin{equation}\label{C1}
f\,\, \textrm{is increasing},\,\, f(0)>0 \,\, \textrm{and superlinear at infinity:} \lim_{s\to \infty}\frac{f(s)}{s}=\infty.
\end{equation}

Gelfand-type problem appears in a number of applications. For instance, it is worth mentioning that Lord Kelvin proposed that such an equation describes an isothermal ball of gas in gravitational equilibrium (see \cite{Ch}). As well as in thermal explosions \cite{FK} and temperature distribution in an object heated by a uniform electric current \cite{KC}. There is an extensive literature of works on applications related to the problem, among others \cite{S,ZB} and the references therein. The particular case $f(s) = e^s$  was first presented by Barenblatt in a volume edited by
Gelfand \cite{G1}, and was motivated by problems occurring in combustion theory (see also \cite{JL}).

It is well-known (see for instance the monograph \cite{D}) that, under assumption \eqref{C1}, there exists an extremal parameter $\lambda^* \in (0, \infty)$ such that, for $0 \leq \lambda < \lambda^*$,  \eqref{Gelfand0}
admits a minimal classical nonnegative solution $u_\lambda$.  On the other hand, if $\lambda > \lambda^*$, then \eqref{Gelfand0}
has no classical solution. The family of minimal solutions $\{ u_\lambda \, : \,  0 \leq \lambda <\lambda^*  \}$ is increasing in $\lambda$, and its limit as $\lambda \uparrow \lambda^*$ is a weak solution $u^* = u_{\lambda^*}$ of \eqref{Gelfand0} for $\lambda = \lambda^*$. The function $u^*$ is called the {\it extremal solution} of \eqref{Gelfand0}.

 In order to prove the regularity of $u^*$ an important role is played by the stability condition. Concretely, a weak solution of  \eqref{Gelfand0} is called {\it stable} if  the energy functional associated to~\eqref{C1}
$$\mathcal{E}(u):= \int_\Omega \left( \frac12 \vert \nabla u \vert^2 -  \lambda F(u) \right) dx, \quad \hbox{where} \quad F' = f, $$
has
a nonnegative definite second variation. In the $1970s$ Crandall and Rabinowitz \cite{CrR} established that if $f(s) = e^s$ or $f(s) = (1 + s)^p$ with $p >1$
then stable solutions in any smooth bounded domain $\Omega$ are bounded (and hence smooth and
analytic, by classical elliptic regularity theory) when $n \leq  9$ and therefore $u^*$ is bounded.
The last twenty-five years, for a certain type of non-linearities $f$, have produced a large literature on Gelfand-type problems. See the monograph \cite{D} for an extensive list of results and references. The main developments proving that stable solutions to \eqref{Gelfand0}  are smooth (no matter what the nonlinearity $f$ is) were made by different authors, see, for example,  \cite{N}, \cite{CC}, \cite {V}, \cite{CR} (see also the survey by Cabr\'{e} \cite{C}). The above results give some partial answer to the an open question by Brezis \cite{BrezisOP} who asks if it possible in dimension $n \leq 9$ to construct some $f$ and some $\Omega$ for which a singular stable solution exists.
 Cabr\'{e}, Figalli, Ros-Oton and  Serra  (\cite{CFRS}) finally solved the problem, by establishing the regularity of stable solutions to \eqref{Gelfand0}   in the interior of any open set $\Omega$ in the optimal dimensions $n \leq 9$ under the only
requirement for the nonlinearity $f$ to be nonnegative; furthermore, adding the vanishing
boundary condition $u = 0$ on $\partial \Omega$, they proved regularity up to the boundary when $\Omega$ is
of class $C^3$ and $n \leq  9$, and when $f$ is nonnegative, nondecreasing and convex.

 In this framework, a larger classes of operators are considered; we highlight the $p$-Laplacian, see \cite{CFM,CS1,GAP,SAN, SAN1} and the references therein;  the Laplacian with a quadratic gradient term \cite{ACM,M}; the $1$-homogeneous $p$-Laplacian \cite{CMR}; the $1$-Laplacian \cite{MS} and the $k$-Hessian \cite{J,JS}.

For  nonlocal operators, only in the case of singular kernels there are results about   Gelfand-type problems. For instance,  for the fractional Laplacian $(-\Delta )^s$, there are some results in the line of the above classical result where they come into play the dimension $n$ and also $s$ (see for instance \cite{RS1} and \cite{R1}). Now, to our knowledge, there are no results for the case nonlocal problems with integrable kernels.

Our aim is to study the  Gelfand-type problems in the framework of random walk spaces, which covers the cases of  weighted graphs and of   nonlocal problems with integrable  kernels. Like for the local case we prove that   if $f$ satisfies \eqref{C1} then there exists an extremal parameter $\lambda^* \in (0, \infty)$ such that, for $0 \leq \lambda < \lambda^*$ admits a minimal bounded solution $u_\lambda$, and there are not solution for $\lambda > \lambda^*$. Moreover, assuming $f$ is strictly convex,  the problem admits a minimal bounded solution for $\lambda = \lambda^*$, i.e, a extremal solution. Contrary to what happens in the local case, any solution of the problem is bounded from above and below for two functions that only depends of $f$ and $\lambda$, independently of the random walk space involved.

We also study the stability of the solutions, showing that for $0 \leq \lambda < \lambda^*$ the minimal bounded solutions $u_\lambda$ are stable; and, in the case that $f$ is strictly convex, we show that they are the only stable solutions.

We give simple but   illustrative examples of the many situations that can occur for the solutions of Gelfand-type problems on weighted graphs. For instance we give an example of a simple weighted graph and a convex function $f$ for which there exists infinitely many extremal solutions and also a simple  weighted graph and a non-convex function $f$  for which the function  $\lambda\mapsto ||u_\lambda||_\infty$ is not continuous.

\section{Preliminaires}
\setcounter{equation}{0}

 \subsection{Random walk spaces}  We recall some concepts and results about random walk spaces  given in \cite{MST0},  \cite{MST2} and \cite{MSTBook}.

 Let $(X,\mathcal{B})$ be a measurable space such that the $\sigma$-field $\mathcal{B}$ is countably generated.
A random walk $m$
on $(X,\mathcal{B})$ is a family of probability measures $(m_x)_{x\in X}$
on $\mathcal{B}$ such that $x\mapsto m_x(B)$ is a measurable function on $X$ for each fixed $B\in\mathcal{B}$.

The notation and terminology chosen in this definition comes from Ollivier's paper~\cite{O}. As noted in that paper, geometers may think of $m_x$ as a replacement for the notion of balls around $x$, while in probabilistic terms we can rather think of these probability measures as defining a Markov chain whose transition probability from $x$ to $y$ in $n$ steps is
$$
\displaystyle
dm_x^{*n}(y):= \int_{z \in X}  dm_z(y)dm_x^{*(n-1)}(z), \ \ n\ge 1
$$
and $m_x^{*0} = \delta_x$, the dirac measure at $x$.

\begin{definition} {\rm
If $m$ is a random walk on $(X,\mathcal{B})$ and $\mu$ is a $\sigma$-finite measure on $X$. The convolution of $\mu$ with $m$ on $X$ is the measure defined as follows:
$$\mu \ast m (A) := \int_X m_x(A)d\mu(x)\ \ \forall A\in\mathcal{B},$$
which is the image of $\mu$ by the random walk $m$.}
\end{definition}

\begin{definition}  {\rm If $m$ is a random walk on $(X,\mathcal{B})$,
a $\sigma$-finite measure $\nu$ on $X$ is {\it invariant}
with respect to the random walk $m$ if
 $$\nu\ast m = \nu.$$

The measure $\nu$ is said to be {\it reversible} if moreover, the detailed balance condition $$dm_x(y)d\nu(x)  = dm_y(x)d\nu(y),$$ holds true.}
\end{definition}

\begin{definition} {\rm
Let $(X,\mathcal{B})$ be a measurable space where the $\sigma$-field $\mathcal{B}$ is countably generated. Let $m$ be a random walk on $(X,\mathcal{B})$ and $\nu$ an invariant measure with respect to $m$. The measurable space together with $m$ and $\nu$ is then called a random walk space
and is denoted by $[X,\mathcal{B},m,\nu]$.}
\end{definition}

 If $(X,d)$ is a Polish metric space (separable completely metrizable topological space), $\mathcal{B}$ is its Borel $\sigma$-algebra and $\nu$ is a Radon measure (i.e. $\nu$ is inner regular
and locally finite), then we denote  $[X,\mathcal{B},m,\nu]$ as $[X,d,m,\nu]$,   and call it a metric random walk space.

\begin{definition} {\rm
We say that a random walk space $[X, \mathcal{B}, m, \nu]$ is $m$-connected
if, for every $D\in \mathcal{B}$ with $\nu(D)>0$ and $\nu$-a.e. $x\in X$,
$$\sum_{n=1}^{\infty}m_x^{\ast n}(D)>0.$$
}
\end{definition}

 Formally, the idea behind the above definition is that all the parts of the space considered as $m$-connected can be reached after certain number of jumps (induced by the random walk $m$), no matter what the starting point is, except for, at most, a $\nu$-null set of points.

 \begin{definition} {\rm
Let  $[X,\mathcal{B},m,\nu]$ be a random walk space and let $A$, $B\in\mathcal{B}$. We define the {\it $m$-interaction} between $A$ and $B$ as
$$ L_m(A,B):= \int_A \int_B dm_x(y) d\nu(x)=\int_A m_x(B) d\nu(x).
 $$
 }
 \end{definition}

  The following result gives a characterization of $m$-connectedness in terms of the $m$-interaction between sets,  and provides more intuition for the concept and terminology used. In~\cite{MSTBook} one can find its relation with $\nu$-essential irreducibility of a random walk $m$ between other characterizations, we do not enter here in such details. In Section~\ref{nonlocal.notions.1.section} we introduce the  $m$-Laplacian and give another characterization  via the ergodicity of such operator that will be used later on.

\begin{proposition}(\cite[Proposition 2.11]{MST0}, \cite[Proposition 1.34]{MSTBook})
 Let $[X,\mathcal{B},m,\nu]$ be a random walk space.
The following statements are equivalent:
\item{ (i) } $[X,\mathcal{B},m,\nu]$ is $m$-connected.
\item {(ii)} If $ A,B\in\mathcal{B}$ satisfy $A\cup B=X$ and $L_m(A,B)= 0$, then either $\nu(A)=0$ or $\nu(B)=0$.
\end{proposition}

\begin{definition}
Let $[X,\mathcal{B},m,\nu]$ be a reversible random walk space, and let $\Omega\in\mathcal{B}$ with $\nu(\Omega)>0$. We denote by  $\mathcal{B}_\Omega$ to the following $\sigma$-algebra
 $$\mathcal{B}_\Omega:=\{B\in\mathcal{B} \, : \, B\subset \Omega\}.$$
 We  will denote by $\nu \res \Omega$ the restriction of $\nu$ to $\mathcal{B}_\Omega$.

We say that $\Omega$ is {\it $m$-connected
(with respect to $\nu$)} if $L_m(A,B)>0$ for every pair of non-$\nu$-null   sets $A$, $B\in \mathcal{B}_\Omega$ such that $A\cup B=\Omega$.
\end{definition}

Let us see now some examples of random walk spaces.

 \begin{example}\label{example.nonlocalJ} \rm
Consider the metric measure space $(\R^N, d, \mathcal{L}^N)$, where $d$ is the Euclidean distance and $\mathcal{L}^N$ the Lebesgue measure on $\R^N$ (which we will also denote by $|.|$). For simplicity, we will write $dx$ instead of $d\mathcal{L}^N(x)$. Let  $J:\R^N\to[0,+\infty[$ be a measurable, nonnegative and radially symmetric
function  verifying  $\int_{\R^N}J(x)dx=1$. Let $m^J$ be the following random walk on $(\R^N,d)$:
$$m^J_x(A) :=  \int_A J(x - y) dy \quad \hbox{ for every $x\in \R^N$ and every Borel set } A \subset  \R^N  .$$
Then, applying Fubini's Theorem it is easy to see that the Lebesgue measure $\mathcal{L}^N$ is reversible with respect to $m^J$. Therefore, $[\R^N, d, m^J, \mathcal{L}^N]$ is a reversible metric random walk space.

 Observe that a domain in $\R^N$ is $m^J$-connected, and we can also have disjoint domains $A$ and $B$ whose union is $m^J$-connected if  $L_{m^J}(A,B)>0$.  \hfill $\blacksquare$
\end{example}

\begin{example}\label{example.graphs}[Weighted discrete graphs] \rm Consider a locally finite  weighted discrete graph $$G = (V(G), E(G)),$$ where $V(G)$ is the vertex set, $E(G)$ is the edge set and each edge $(x,y) \in E(G)$ (we will write $x\sim y$ if $(x,y) \in E(G)$) has a positive weight $w_{xy} = w_{yx}$ assigned. Suppose further that $w_{xy} = 0$ if $(x,y) \not\in E(G)$.  Note that there may be loops in the graph, that is, we may have $(x,x)\in E(G)$ for some $x\in V(G)$ and, therefore, $w_{xx}>0$. Recall that a graph is locally finite if every vertex is only contained in a finite number of edges.

 A finite sequence $\{ x_k \}_{k=0}^n$  of vertices of the graph is called a {\it  path} if $x_k \sim x_{k+1}$ for all $k = 0, 1, ..., n-1$. The {\it length} of a path $\{ x_k \}_{k=0}^n$ is defined as the number $n$ of edges in the path. We will assume that  $G = (V(G), E(G))$ is {\it connected}, that is, for any two vertices $x, y \in V$, there is a path connecting $x$ and $y$, that is, a path $\{ x_k \}_{k=0}^n$ such that $x_0 = x$ and $x_n = y$.  Finally, if $G = (V(G), E(G))$ is connected, the {\it graph distance} $d_G(x,y)$ between any two distinct vertices $x, y$ is defined as the minimum of the lengths of the paths connecting $x$ and $y$. Note that this metric is independent of the weights.

For $x \in V(G)$ we define the weight at $x$ as
$$d_x:= \sum_{y\sim x} w_{xy} = \sum_{y\in V(G)} w_{xy},$$
and the neighbourhood of $x$ as $N_G(x) := \{ y \in V(G) \, : \, x\sim y\}$. Note that, by definition of locally finite graph, the sets $N_G(x)$ are finite. When all the weights are~$1$, $d_x$ coincides with the degree of the vertex $x$ in a graph, that is,  the number of edges containing $x$.

For each $x \in V(G)$  we define the following probability measure
$$m^G_x:=  \frac{1}{d_x}\sum_{y \sim x} w_{xy}\,\delta_y.\\ \\
$$
It is not difficult to see that the measure $\nu_G$ defined as
 $$\nu_G(A):= \sum_{x \in A} d_x,  \quad A \subset V(G),$$
is a reversible measure with respect to the random walk $m^G$. Therefore, we have that $[V(G),\mathcal{B},m^G,\nu_G]$ is a reversible random walk space being $\mathcal{B}$ is the $\sigma$-algebra of all subsets of $V(G)$. Moreover $[V(G),d_G,m^G,\nu_G]$ is a reversible metric random walk space. Observe also that in this case,    the connectedness assumed for  $G$ is equivalent to the $m^G$-connectedness of $[V(G),d_G,m^G,\nu_G]$  (this is in contrast with the situation in Example~\ref{example.nonlocalJ}). \hfill $\blacksquare$
\end{example}

\begin{example} \rm Given a random walk  space $[X,\mathcal{B},m,\nu]$ and $\Omega \in \mathcal{B}$ with $\nu(\Omega)>0$, let
$$m^{\Omega}_x(A):=\int_A d m_x(y)+\left(\int_{X\setminus \Omega}d m_x(y)\right)\delta_x(A) \quad \hbox{ for every } A\in\mathcal{B}_\Omega  \hbox{ and } x\in\Omega.
$$
Then, $m^{\Omega}$ is a random walk on $(\Omega,\mathcal{B}_\Omega)$ and it easy to see that $\nu \res \Omega$   is invariant with respect to $m^{\Omega}$. Therefore,  $[\Omega,\mathcal{B}_\Omega,m^{\Omega},\nu \res \Omega]$ is a random walk space. Moreover, if $\nu$ is reversible with respect to $m$ then $\nu \res \Omega$ is  reversible with respect to $m^{\Omega}$. Of course, if $\nu$ is a probability measure we may normalize $\nu \res \Omega$ to obtain the random walk space
$$\left[\Omega,\mathcal{B}_\Omega,m^{\Omega}, \frac{1}{\nu(\Omega)}\nu \res \Omega \right].$$
 Note that, if $[X,d,m,\nu]$ is a metric random walk space and $\Omega$ is closed, then we have that $[\Omega,d,m^{\Omega},\nu \res \Omega]$ is also a metric random walk space, where we abuse notation and denote by $d$ the restriction of $d$ to $\Omega$.

In particular, in the context of Example \ref{example.nonlocalJ}, if $\Omega$ is a closed and bounded subset of $\R^N$, we obtain the metric random walk space $[\Omega, d, m^{J,\Omega},\mathcal{L}^N\res \Omega]$ where
$m^{J,\Omega} := (m^J)^{\Omega}$; that is,
$$m^{J,\Omega}_x(A):=\int_A J(x-y)dy+\left(\int_{\R^n\setminus \Omega}J(x-z)dz\right)d\delta_x,$$ for every Borel set   $A \subset  \Omega$  and $x\in\Omega$.
\hfill $\blacksquare$
\end{example}

 Another important  example of random walk spaces are  weighted hypergraphs (see \cite[Example 1.44]{MSTBook}).

\subsection{Nonlocal operators}\label{nonlocal.notions.1.section}

Let us introduce the nonlocal counterparts of some classical concepts  and operators that will play a role in the paper.

\begin{definition}{\rm
Let $[X,\mathcal{B},m,\nu]$ be a random walk space. Given a function $f: X \rightarrow \R$ we define its {\it nonlocal gradient}
$\nabla f: X \times X \rightarrow \R$ as
$$\nabla f (x,y):= f(y) - f(x) \quad \forall \, x,y \in X.$$

Moreover, given $\z : X \times X \rightarrow \R$, its {\it $m$-divergence}
${\rm div}_m \z : X \rightarrow \R$ is defined, when it has sense, as
 $$({\rm div}_m \z)(x):= \frac12 \int_{X} (\z(x,y) - \z(y,x)) dm_x(y).$$
 }
\end{definition}

  For the nonlocal gradient we have the following {\it Leibnitz formula}:
$$
\nabla (fg) (x,y) = \nabla f(x,y) g(x) +\nabla g(x,y)f(y) =  \nabla f(x,y)g(y)+ \nabla g(x,y)f(x) .
$$
We also have
$$\nabla \left(\frac{1}{g} \right)(x,y) = - \frac{\nabla g(x,y)}{g(x) g(y)}.$$
Therefore, we have
$$
\begin{array}{rl}\displaystyle
\nabla \left(\frac{f^2}{g}\right) (x,y)\!\!\! & \displaystyle = \frac{\nabla f(x,y)(f(x) + f(y))g(x) - f^2(x) \nabla g(x,y)}{g(x) g(y)}\\ \\
\displaystyle &\displaystyle  = \frac{\nabla f(x,y)(f(x) + f(y))g(y) - f^2(y) \nabla g(x,y)}{g(x) g(y)}.
\end{array}
$$
 Observe that also
 \begin{equation}\label{Leib11dos}
\nabla \left(\frac{f^2}{g}\right) (x,y)  = \frac{\nabla f(x,y)(f(x) + f(y))\frac{g(x)+g(y)}{2} - \frac{f^2(x)+f^2(y)}{2} \nabla g(x,y)}{g(x) g(y)}.
\end{equation}

We define the (nonlocal) Laplace operator as follows.
\begin{definition}{\rm
Let $[X,\mathcal{B},m,\nu]$ be a random walk space, we define the {\it $m$-Laplace operator} (or {\it $m$-Laplacian}) from $L^1(X,\nu)$ into itself as $\Delta_m:= M_m - I$, i.e.,
$$\Delta_m f(x)= \int_X f(y) dm_x(y) - f(x) = \int_X (f(y) - f(x)) dm_x(y), \quad x\in X,$$
for $f\in L^1(X,\nu)$.
}
\end{definition}
 Note that
$$\Delta_m f (x) = {\rm div}_m (\nabla f)(x).$$

 From the reversibility of $\nu$ respect to $m$, we have the following {\it integration by parts formula}
$$\begin{array}{ll}
 -\displaystyle\int_{\Omega_m}\int_{\Omega_m}  \Delta f(x,y)  dm_x(y)g(x) d\nu(x) \\[10pt] = \displaystyle\frac12 \int_{\Omega_m}\int_{\Omega_m} \nabla f(x,y) \nabla g(x,y) dm_x(y) d\nu(x), \end{array}
$$
 if $f, g \in L^2(\Omega_m, \nu)$.w

In the case of the random walk space associated with a locally finite weighted discrete graph $G=(V,E)$ (as defined in Example~\ref{example.graphs}), the $m^G$-Laplace operator coincides with the graph Laplacian (also called the normalized graph Laplacian) studied by many authors (see, for example, \cite{BJ}, \cite{BJL}, \cite{Elmoatazetal}, \cite{AGrigor}, ):
$$\Delta_{m^G} u(x):=\frac{1}{d_x}\sum_{y\sim x}w_{xy}(u(y)-u(x)), \quad u\in L^2(V,\nu_G), \ x\in V .$$
 In the case of the random walk space $m^J$ of   Example \ref{example.nonlocalJ}, we have
$$\Delta_{m^J} u(x) = \int_{\R^N} J(x-y) (u(y) - u(x)) dy,$$
that coincides with the nonlocal Laplacian studied in \cite{ElLibro}.

\begin{theorem}[\cite{MSTBook}]
Let $[X,\mathcal{B},m,\nu]$ be a random walk space with $\nu(X)<+\infty$ (it can be, in fact, normalized to $1$). Then $[X,\mathcal{B},m,\nu]$ is $m$-connected iff $\Delta_m$ is ergodic, that is,
$$\Delta_mf=0 \ \nu\hbox{-a.e.  implies that $f$ is constant $\nu$-a.e.}$$
\end{theorem}

\subsection{The Lambert $W$ function}

The inverse of the real function $\psi(x) = x e^x$, $x \in \R$, is  called the {\it Lambert $W$ function}. Some examples of problems where the Lambert $W$ function appears can be found in \verb"https://en.wikipedia.org/wiki/Lambert_W_function". The Lambert $W$ function is a mulivalued function because $f$ is not inyective over all $\R$. Nevertheless, $\psi$ is inyective on the intevals $[-\infty, -1[$ and $[-1,+\infty[$. The inverse of $\psi$ on the interval $[-1,+\infty[$ is a single-valued function defined on $[-\frac{1}{e},+\infty[$, denoted by $W_0$ and called  principal branch of the Lambert $W$ function, and the invere on the interval $[-\infty, -1[$ is also a single-valued function defined on $]-\frac{1}{e}, 0[$, denoted by $W_{-1}$ and called the negative branch of the Lambert $W$ function. This function has a special relation with power towers and logarithm towers (see \cite{knuthetal} and~\cite{javiertoledo}.)

 \section{Gelfand-Type Problems in Random Walk Spaces}

\subsection{Assumptions on the spaces and previous results}
Let  $[X,\mathcal{B},m,\nu]$ be a reversible random walk space.
  Given   $\Omega \in \mathcal{B}$,  we define  the {\it $m$-boundary of $\Omega$} by
$$\partial_m\Omega:=\{ x\in X\setminus \Omega \, : \, m_x(\Omega)>0 \}$$
and its {\it $m$-closure} as
$$\Omega_m:=\Omega\cup\partial_m\Omega.$$
  From now on we will assume that  $\Omega$ is $m$-connected (which implies that also  $\Omega_m$  is $m$-connected) and $$0<\nu(\Omega)<\nu(\Omega_m)<\infty.$$

Let us define
 $$L^2_0(\Omega_m, \nu):= \{ u \in L^2(\Omega_m, \nu)  :  u=0 \ \nu\hbox{-a.e. in }  \partial\Omega_m \}.$$
We will assume that $\Omega$ satisfies  a {\it $2$-Poincar\'{e} inequality}, that is,
there exists $\lambda >0$ such that
$$
\lambda \int_\Omega \vert u (x) \vert^2 d\nu(x)\leq   \frac{1}{2} \int_{\Omega_m \times \Omega_m} \vert \nabla u(x,y) \vert^2 d(\nu\otimes m_x)(x,y)
$$
for all $u \in L^2_0(\Omega_m, \nu)$.
This is equivalent to say that
\begin{equation}\label{Realyq} 0<\lambda_{m}(\Omega)  := \inf_{u\in L^2_0(\Omega_m, \nu)\setminus\{0\} }\frac{\displaystyle
\frac12\int_{\Omega_m}\int_{\Omega_m} |\nabla u(x,y)|^2 dm_x(y)d\nu(x)}{\displaystyle\int_\Omega u(x)^2 d\nu(x) }\end{equation}

Join to assumption~\eqref{Realyq} we also assume that
\begin{equation}\label{eigenproblem}\hbox{ exists a non-null } \varphi_{m}\in L^2(\Omega_m,\nu): \ \left\{\begin{array}{l}
\displaystyle- \Delta_m \varphi_{m} (x)=\lambda_{m}(\Omega)\varphi_{m}(x),\quad \nu\hbox{-a.e.} \ x\in\Omega,
\\ \\
\varphi_{m}(x)=0,\quad x\in\partial_m\Omega.
\end{array}\right.
\end{equation}
  Observe that consequently the infimum defining $\lambda_{m}(\Omega)$ in~\eqref{Realyq} is attained at $\varphi_{m}$, that it is unique up to a multiplicative constant and that we can suppose that it is non-negative. We say that $\lambda_{m}(\Omega)$ is  the first eigenvalue of the $m$-Laplacian with homogeneous Dirichlet boundary conditions with associated eigenfunction $\varphi_{m}$.   Note that, from the equation in \eqref{eigenproblem}, we get
  $$\varphi_{m}(x) = \int_\Omega \varphi_{m}(y) dm_x(y) + \lambda_{m}(\Omega)\varphi_{m}(x)  > \lambda_{m}(\Omega)\varphi_{m}(x),\quad \nu\hbox{-a.e.} \ x\in\Omega.$$
  Thus,
  \begin{equation}\label{ok}
  0 < \lambda_{m}(\Omega) < 1.
  \end{equation}
We also have that
\begin{equation}\label{est001}
(1-\lambda_{m}(\Omega))\varphi_{m}(x)=\int_\Omega \varphi_{m}(y)dm_x(y),\quad x\in\Omega.
\end{equation}
Then, by \eqref{ok}, we can write~\eqref{est001} as
\begin{equation}\label{est001N} \varphi_{m}(x)=\frac{1}{1-\lambda_{m}(\Omega)}\int_\Omega \varphi_{m}(y)dm_x(y).
\end{equation}

We moreover assume that {\it there exists an eigenfunction $\varphi_{m}$ associated to   $\lambda_{m}(\Omega)$ such that}
\begin{equation}\label{condH}  \hbox{ $\alpha_\Omega \le \varphi_{m}\le M_\Omega$ in $\Omega$, for some constants $\alpha_\Omega, M_\Omega>0$.}
\end{equation}
  Observe that we are assuming~\eqref{Realyq}, \eqref{eigenproblem} and~\eqref{condH} jointly.

\begin{remark}\rm
\item{ 1.}  Let $[\R^N, d, m^J, \mathcal{L}^N]$  be the metric random walk space   given in Example~\ref{example.nonlocalJ} with $J$ continuous  and compactly supported. For $\Omega$ a bounded domain,
     the assumption~\eqref{condH}  is true,  see~\cite[Section 2.1.1]{ElLibro}.

\item{ 2.}
 For weighted  discrete graphs,   $\lambda_{m^G,1}(\Omega)$ is an eigenvalue with $0<\lambda_{m^G,1}(\Omega)\le 1$ (see~\cite{AGrigor}).    Now,   since we are assuming that $\Omega$ is $m^G$-connected,
 $0<\lambda_{m^G,1}(\Omega)< 1.$
And, by connectedness, using~\eqref{est001N}, we have that~\eqref{condH}  is also true.
\hfill$\blacksquare$
\end{remark}

  Let $\g : \Omega \times \R \rightarrow \R$   a $(\nu \times \mathcal{L}^1)$-measurable function  such that,  for $\nu$-almost all $x \in \Omega_m$, the function $t \mapsto \g(x,t)$ is continuous.
  Consider the problem
$$
\qquad\qquad\left\{\begin{array}{ll} - \Delta_m u = \g(x,u(x)), \quad&\hbox{in} \ \Omega, \\[10pt] u =0, \quad&\hbox{on} \ \partial_m\Omega.  \end{array} \right.\qquad\qquad (P(\g))
$$

 \begin{definition}{\rm We say that a function $\underline{u} \in L^2(\Omega_m, \nu)$ is a {\it subsolution} of  if $(P(\g))$ if verifies
 $$\qquad\qquad\left\{\begin{array}{ll} - \Delta_m \underline{u}(x)  \leq \g(x,\underline{u}(x)) \quad&x \in \Omega, \\[10pt] \underline{u}  \leq 0 \quad&\hbox{on} \ \partial_m\Omega.  \end{array} \right. $$
 We say that a function $\overline{u} \in L^2(\Omega_m, \nu)$ is a {\it supersolution} of  if $(P(\g))$ if verifies
 $$\qquad\qquad\left\{\begin{array}{ll} - \Delta_m \overline{u}(x) \geq \g(x,\overline{u}(x)) \quad&\ x \in \Omega, \\[10pt] \overline{u} \geq 0 \quad&\hbox{on} \ \partial_m\Omega.  \end{array} \right. $$
 We say that a function $u \in L_0^2(\Omega_m, \nu)$ is a {\it solution} of  if $(P(\g))$ if verifies
 $$\qquad\qquad\left\{\begin{array}{ll} - \Delta_m u(x) = \g(x,u(x)) \quad&\ x \in \Omega, \\[10pt] u = 0 \quad&\hbox{on} \ \partial_m\Omega,  \end{array} \right. $$
 that is when $u$ is a subsolution and a supersolution.
 }
 \end{definition}

Consider the {\it energy functional} $\mathcal{E}_{\g} : L_0^2(\Omega_m, \nu) \rightarrow \R$ associate associated with problem $(P(\g))$, i.e.,
$$\mathcal{E}_{\g}(u):= \frac{1}{4 } \int_{\Omega_m \times \Omega_m} \vert \nabla u(x,y) \vert^2 d(\nu\otimes m_x)(x,y) - \int_{\Omega} G(x,u(x)) d\nu(x),$$
with $$G(x,s):= \int_0^s \g(x,t)\, dt.$$
 If $g :  \R \rightarrow \R$ is continuous, we also write $\mathcal{E}_{g}$ and $(P(\g))$ for $\mathcal{E}_{\g}$ and $(P(\g))$ with $\g(x,s)=g(s)$; in this case
$\displaystyle G(x,s)=\int_0^s g(t)\, dt.$
 If $\varphi \in L^2(\Omega, \nu)$, we also write $\mathcal{E}_{\varphi}$ and $(P(\varphi))$ for $\mathcal{E}_{\g}$ and $(P(\g))$ with $\g(x,s)=\varphi(x)$; in this case,
 $\displaystyle G(x,s)=\varphi(x)s.$

 The following result is easy to prove.

\begin{lemma} Assume that
\begin{equation}\label{lqn01}\hbox{$\exists \, g_0 \in L^2(\Omega,\nu)$ such that $\vert \g(x,s) \vert \leq g_0(x)$ $\nu$-a.e and for all $s \in \R$.}
\end{equation}
For $u \in L_0^2(\Omega_m, \nu)$ the Fr\'{e}chet derivative of $\mathcal{E}_{\g}$ at $u$ is given by
$$\langle D\mathcal{E}_{\g}(u), v \rangle = \frac12\int_{\Omega_m \times \Omega_m} \nabla u (x,y) \cdot \nabla v(x,y)d(\nu\otimes m_x)(x,y) - \int_\Omega \g(x,u(x)) v(x) d\nu(x),
$$
$v \in L_0^2(\Omega_m, \nu)$.
\end{lemma}

By the integration by parts formula, we have the following result.

\begin{lemma}\label{EnerSol}  Assume~\eqref{lqn01}. We have that
$u \in L_0^2(\Omega_m, \nu)$ is a  solution of problem $(P(\g))$ if and only if $u$ is a critical point of $\mathcal{E}_{\g}$,  that is $ D\mathcal{E}_{\g}(u)=0$.
\end{lemma}

Like for the local case, the following  Maximum Principle  will play an important role.

\begin{lemma}  Given $u \in L^2(\Omega_m, \nu)$, if $u$ satisfies
$$
 \left\{\begin{array}{ll} - \Delta_m u  \leq 0 \quad& x\in \Omega, \\[10pt] u(x) \leq 0 \quad&x \in \partial_m\Omega,  \end{array} \right.
$$
 then
 $$u(x) \leq 0 \quad  \nu-a.e. \ \  \ x \in \Omega_m.$$
\end{lemma}
\begin{proof} Multiplying by $u^+$ and integrating in $\Omega$,  having in mind that $u^+ = 0$ $\nu$-a.e. in $\partial_m \Omega$, and integrating by parts, we get
$$0 \geq - \int_{\Omega_m}  \int_{\Omega_m} (u(y) - u(x)) u^+(x) \, dm_x(y) \, d \nu(x) $$ $$= \frac12 \int_{\Omega_m}  \int_{\Omega_m} (u(y) - u(x))( u^+(y) - u^+(x))\, dm_x(y) \, d \nu(x)$$ $$\ge \frac12 \int_{\Omega_m}  \int_{\Omega_m} (u^+(y) - u^+(x))^2\, dm_x(y) \, d \nu(x).$$
Then, since $u^+ \in L^2_0(\Omega_m, \nu)$, applying Poincar\'{e}'s inequality, we get $u^+ =0$ $\nu$-a.e. in $\Omega$, and the proof finishes.
\end{proof}

\begin{remark} \rm Note that the Maximum Principle says that if $\underline{u}$ is a subsolution of problem $(P(0))$, then $\underline{u} \leq 0$. Obviously, we also have that  if $\overline{u}$ is a supersolution of problem $(P(0))$, then $\overline{u} \geq 0$.\hfill$\blacksquare$
\end{remark}

Given $\varphi \in L^2(\Omega, \nu)$,  consider the problem $(P(\varphi))$:
\begin{equation}\label{Dirichlet1}
\left\{\begin{array}{ll}
-\Delta_m u  =\varphi&\hbox{in } \Omega,\\[10pt]
u(x)=0&\hbox{on }\partial_m\Omega.
\end{array}\right.
\end{equation}
Existence and  uniqueness for Problem~\eqref{Dirichlet1} are  given in \cite{ST0} (see also~\cite{MSTBook}). Nevertheless, and for the sake of completeness, we give  the next result with a  different proof.
\begin{theorem}\label{EUp}
Given $\varphi \in L^2(\Omega, \nu)$,
there is a unique solution $u$ of Problem~\eqref{Dirichlet1}
Moreover, $u$ is the only minimizer of the variational problem
$$
 \min_{h \in  L_0^2(\Omega_m, \nu) \setminus \{ 0 \}} \mathcal{E}_{\varphi}(h).
$$
\end{theorem}
\begin{proof}    First note that $\mathcal{E}_{\varphi}$  is convex and lower semicontinuous in $L^2(\Omega,\nu)$, thus weakly lower semicontinuous (see~\cite[Corollary 3.9]{BrezisAF}).
Set $$\theta := \inf_{h \in L_0^2(\Omega_m, \nu) \setminus \{ 0 \}} \mathcal{E}_{\varphi}(h),$$ and let $\{ u_n \}$ be a minimizing sequence, hence
$$\theta = \lim_{n \to \infty} \mathcal{E}_{\varphi}(u_n).$$  Set $$  K:= \sup_{n \in \N}\mathcal{E}_{\varphi}(u_n)  < + \infty\,.$$
Since $\Omega_m$ satisfies a $2$-Poincar\'{e} type inequality, by Young's inequality, we have
\begin{equation}\label{lqsr00}
\begin{array}{c}\displaystyle \lambda_{m}(\Omega)\int_\Omega \left|u_n(x) \right|^2 \, {d\nu(x)} \leq  \frac{1}{2}\int_{\Omega_m \times \Omega_m} \vert \nabla u_n(x,y) \vert^2 d(\nu\otimes m_x)(x,y)
\\ \\
\displaystyle=  2 \mathcal{E}_{\varphi}(u_n) + 2 \int_{\Omega} \varphi(x) u_n(x) d \nu(x)
\\\displaystyle \leq   2K + \lambda_{m}(\Omega)^2 \int_\Omega \left|u_n(x) \right|^2 \, {d\nu(x)} +  \frac{1}{\lambda_{m}(\Omega)^2} \int_\Omega \vert \varphi(x) \vert^2 d\nu(x).
\end{array}
\end{equation}
Therefore,  since $\lambda_{m}(\Omega)<1$ (see \eqref{ok}),  we obtain that
\begin{equation}\label{lqsr02}\int_{\Omega} |u_n(x)|^2 \, {d\nu(x)} \leq C \quad \forall n \in \N.
\end{equation}
Hence, up to a subsequence, we have
$$u_n \rightharpoonup u \ \ \ \hbox{in }   L_0^2(\Omega_m, \nu).$$
Furthermore, using the weak lower semicontinuity of the functional $\mathcal{E}_{g}$, we get
$$\mathcal{E}_{\varphi}(u) = \inf_{h \in L_0^2(\Omega_m, \nu) \setminus \{ 0 \}} \mathcal{E}_{\varphi}(h).$$
  Then,   $u$ is a minimizer of $\mathcal{E}_{\varphi}$ and hence   a critical point. Therefore,  by Lemma \ref{EnerSol}, $u$ is a   solution of problem~\eqref{Dirichlet1}.
  Since the functional $\mathcal{E}_{g}$ is strictly convex, we have that $u$ is its unique critical point, and uniqueness follows.
\end{proof}

 The following result is a nonlocal variant of ~\cite[Lemma 1.1.1]{D}.

  \begin{theorem}\label{Truncat} Let $g \in C^1(\R)$ be and assume that $\underline{u}$ and $\overline{u}$ are   $L^\infty$-bounded subsolution and supersolution, respectively, of Problem $(P(g))$   with and $\underline{u} \leq \overline{u}$. Then, if
$$\h(x,r):= \left\{ \begin{array}{lll} g(\underline{u}(x)) \quad &\hbox{if} \ \ r < \underline{u}(x), \\[10pt] g(r) \quad &\hbox{if} \ \ \underline{u}(x) \leq r \leq \overline{u}(x) \\[10pt]  g(\overline{u}(x)) \quad &\hbox{if} \ \ r >\underline{u}(x), \end{array} \right.$$
then there is a  solution $u$ of problem  $(P(\h))$.
Moreover, if $g$ is non-decreasing, we have $ \underline{u}\leq u \leq \overline{u} $, and then $u$ is solution of $(P(g))$.
\end{theorem}
\begin{proof}  The proof is similar to that of Theorem~\ref{EUp}. The main fact is to get~\eqref{lqsr00} and~\eqref{lqsr02}  which follow in this way: Since $\Omega_m$ satisfies a $2$-Poincar\'{e} type inequality, by  Young's inequality, given $\epsilon >0$,  we have
$$\displaystyle \lambda_{m}(\Omega)\int_\Omega \left|u_n(x) \right|^2 \, {d\nu(x)} \leq  \frac{1}{2}\int_{\Omega_m \times \Omega_m} \vert \nabla u_n(x,y) \vert^2 d(\nu\otimes m_x)(x,y)
$$
$$\leq 2 \mathcal{E}_{h}(u_n) + 2 \left\vert \int_{\Omega} H(x,u_n(x)) d \nu(x)\right\vert \leq  2K +  2 \Vert g \Vert_{L^\infty[a,b]} \int_\Omega \vert u_n(x) \vert \, d\nu(x)$$ $$\leq  2K +2 \Vert g \Vert_{L^\infty[a,b]} \left( \epsilon \int_\Omega \left|u_n(x) \right|^2 \, {d\nu(x)} + \frac{1}{\epsilon} \nu(\Omega)^2 \right).$$
Thus, taking $\epsilon$ small enough, we get
$$\int_{\Omega} |u_n(x)|^2 \, {d\nu(x)} \leq C \quad \forall n \in \N.$$

 Suppose now that $g$ is non-decreasing. Then, if $v:= u- \overline{u}$, we have
 $$- \Delta_m v = \h(x,u(x)) - g(\overline{u}(x))  \leq 0 \quad \hbox{for} \ x \in \Omega$$
 and
 $$v(x)  \leq 0 \quad \quad \hbox{for} \ x \in \partial_m\Omega.$$
 Then, by the Maximum Principle, we have $v\le 0$, so $u\le \overline{u}$. Similarly, one can show that $\underline{u}\le u $.  Hence $\h(x,u(x))=g(u(x))$ and $u$ is solution of $(P(g))$.
\end{proof}

\subsection{Gelfand-type problems}

 For $\lambda \geq 0$, we will consider the problem
$$
(P(\lambda f)) \qquad\qquad\left\{\begin{array}{ll} - \Delta_m u = \lambda f(u) \quad&\hbox{in} \ \Omega, \\[10pt] u =0 \quad&\hbox{on} \ \partial_m\Omega,  \end{array} \right.
$$
 where   $f:[0,+\infty)\to \mathbb{R}$ satisfies
 \begin{equation}\label{H}\left\{\begin{array}{l}
(a)\quad f\  \hbox{ is non-decreasing},\\[6pt]
(b)\quad  f(0)>0,\\[6pt]
(c)\quad  \exists \, c_1>0:
f(s) \ge c_1\, s,\ \forall s \geq 0.
\end{array}\right.
\end{equation}
A paradigmatic example is  $f(s)=e^s$ for which
$$e^s\ge es, \quad\forall s>0,$$
being $e$ optimal.

 \begin{definition} We say that $u_\lambda$ is a minimal solution to problem~$(P(\lambda f))$ if $u_\lambda $ is solution and $u_\lambda \leq v$ for any other   solution $v$ to~$(P(\lambda f))$.
 \end{definition}

 \begin{theorem}\label{existence}
 Assume that $f$ satisfies hypotheses~\eqref{H}.
There exists a positive number (called the extremal parameter) $\lambda^*\le \frac{\lambda_{m}(\Omega)}{c_1}$, with $c_1$ satisfying~\eqref{H}-$(c)$,  such that:
\item  {\rm (a)}  If $\lambda > \lambda^*$, the problem~$(P(\lambda f))$ admits no solution.
\item {\rm (b)} If $0\leq \lambda < \lambda^*$, the problem~$(P(\lambda f))$ admits a minimal bounded solution $u_\lambda$.
\item  {\rm (c)} If $0\le \mu<\lambda<\lambda^*$, then $u_{\mu}<u_\lambda$.
\end{theorem}

\begin{proof} Observe first that for $\lambda=0$, $u=0$ is the unique solution of problem~$(P(\lambda f))$.  And if $\lambda>0$ and $u$ is a solution of problem~$(P(\lambda f))$ then $u$ is non-null.

 {\it Step 1}.
  If problem~$(P(\lambda f))$ has a solution $u$, then,  multiplying equation in~$(P(\lambda f))$  by $\varphi_{m}$, integrating in $\Omega$,  having in mind that $\varphi_{m} = 0$ $\nu$-a.e. in $\partial_m \Omega$, and integrating by parts, we get  (using~\eqref{H}-$(b)$)
$$\frac12 \int_{\Omega_m}  \int_{\Omega_m} (u(y) - u(x))( \varphi_{m}(y) - \varphi_{m}(x))\, dm_x(y) \, d \nu(x) = \int_\Omega \lambda f(u) \varphi_{m} d\nu \ge \lambda \,c_1  \int_\Omega u \varphi_{m} d\nu.$$
Similarly, multiplying by $u$ the equation on \eqref{eigenproblem}, we get
$$\frac12 \int_{\Omega_m}  \int_{\Omega_m} (u(y) - u(x))( \varphi_{m}(y) - \varphi_{m}(x))\, dm_x(y) \, d \nu(x) = \lambda_{m}(\Omega)\int_\Omega u \varphi_{m}d \nu.$$
Therefore, $0 < \lambda \le \frac{\lambda_{m}(\Omega)}{c_1}.$

 {\it Step 2}.  Let us see now  that, for $\lambda>0$ small enough, there exists a solution to problem~$(P(\lambda f))$ that is $L^\infty$-bounded by $M_\Omega$, where $M_\Omega$ is defined in \eqref{condH}.

 Fix $0<k\leq\min\{\alpha_\Omega\,\lambda_{m}(\Omega), \lambda_m(\Omega)/c_1\}$, being $\alpha_\Omega$ defined in \eqref{condH}.
By Theorem \ref{EUp}, we know that there exists a positive solution $\tilde{u}$ to problem
$$\left\{\begin{array}{ll}
\displaystyle- \Delta_m \tilde u =k,\quad &\hbox{in} \ \Omega,
\\ \\
\tilde u=0,\quad &\hbox{on} \  \partial_m\Omega.
\end{array}\right.
$$
Then,
$$- \Delta_m (\tilde{u} - \varphi_{m}) = k - \lambda_{m}(\Omega)\varphi_{m}  < k - \lambda_{m}(\Omega) \alpha_\Omega \leq 0 \quad \hbox{in } \Omega,$$
and
$$\tilde{u} - \varphi_{m} = 0 \quad \hbox{on} \ \partial_m\Omega.$$
Hence, applying the Maximum Principle two times, we have
$$ 0\le \tilde{u} \le \varphi_{m}   \quad \hbox{in} \ \Omega_m,$$
and therefore
\begin{equation}\label{cotalinf}0\le \tilde{u} \leq \varphi_{m} \leq M_\Omega \quad \hbox{in} \ \Omega_m.
\end{equation}

Now, we take
$$
0<\lambda<\frac{k}{f(M_\Omega)}.
$$
 Then, by~\eqref{H}-$(a)$,
$$
- \Delta_m \tilde{u}=k>\lambda\, f(M_\Omega)\geq \lambda \, f(\tilde{u}).
$$

  Again, by  Theorem \ref{EUp}, let $u_0$ be a solution to problem
$$\left\{\begin{array}{ll}
\displaystyle- \Delta_m u_0 =\lambda f(0),\quad &\hbox{in} \ \Omega,
\\ \\
u_0=0,\quad &\hbox{on} \  \partial_m\Omega.
\end{array}\right.
$$
Then, by using~\eqref{H}-$(a)$,
$$
- \Delta_m u_0 =  \lambda f(0)  \le \lambda f(\tilde u)\leq - \Delta_m \tilde{u}.
$$
Hence, by the Maximum Principle, we have
$$
 0\le u_0\leq \tilde{u} \leq M_\Omega
$$
Since $f(u_0) \in L^\infty(\Omega_m,\nu)$, by  Theorem \ref{EUp}, there exists $u_1$ solution to problem
$$\left\{\begin{array}{ll}
\displaystyle- \Delta_m u_1 =\lambda \, f(u_0),\quad &\hbox{in} \ \Omega,
\\ \\
u_1=0,\quad &\hbox{on} \  \partial_m\Omega.
\end{array}\right.
$$
Moreover,
$$
- \Delta_m u_1 =\lambda \, f(u_0) \leq \lambda \, f(\tilde{u})\leq- \Delta_m \tilde{u},
$$
Hence, by the Maximum Principle, we have
$$
0\le u_0\leq u_1\leq \tilde{u} \leq M_\Omega.
$$
Inductively, for $n \in \N$, we find $u_n$ a solution of the problem
$$\left\{\begin{array}{ll}
\displaystyle- \Delta_m u_n =\lambda \, f(u_{n-1}),\quad &\hbox{in} \ \Omega,
\\ \\
u_n=0,\quad &\hbox{on} \  \partial_m\Omega,
\end{array}\right.
$$
and $0\le u_n \leq M_\Omega$ for all $n \in \N$. Then, by the Dominate Convergence Theorem, there exists $u_\lambda= \lim_{n \to \infty} u_n \in L^\infty(\Omega_m, \nu)$ solution to problem~$(P(\lambda f))$.

  In this way, if we define the set
$$
\mathcal{E}:= \{ \lambda>0  :  \exists\, u  \hbox{ bounded solution of   $(P(\lambda f))$}\}
$$
   we have that $$\mathcal{E}\not= \emptyset.$$

{\it Step 3.} {\it $\mathcal{E}$ is an interval.}

    Fix $\mu \in \mathcal{E}\not= \emptyset$ and let $w$ be a bounded solution to problem $(P(\mu f))$. Then, for each $0<\lambda<\mu$,   consider, as before,  solutions to
$$\left\{\begin{array}{ll}
\displaystyle- \Delta_m u_0 =\lambda f(0),\quad &\hbox{in} \ \Omega,
\\ \\
u_0=0,\quad &\hbox{on} \  \partial_m\Omega,
\end{array}\right.
$$
and
$$\left\{\begin{array}{ll}
\displaystyle- \Delta_m u_n =\lambda \, f(u_{n-1}),\quad &\hbox{in} \ \Omega,
\\ \\
u_n=0,\quad &\hbox{on} \  \partial_m\Omega,
\end{array}\right.
$$
for $n\ge 1$.
Now we can argue that $$ 0\le u_n\le w\quad \forall n,$$
and get a solution   $u_\lambda= \lim_{n \to \infty} u_n \in L^\infty(\Omega_m, \nu)$  to problem $(P(\lambda f))$ as before.

Once we know that $\mathcal{E}$ is an interval, we define the extremal parameter
$$\lambda^*:= \sup \{ \lambda   :   \lambda \in \mathcal{E} \}.$$
 By {\it Step 1},
$$\lambda^* \leq \frac{\lambda_{m}(\Omega)}{c_1},$$
and we have that  $$\mathcal{E}\subset  [0,\lambda^*].$$

 The same arguments give that the solutions $u_{\lambda}$   are in fact  minimal; and, if $0\le \mu<\lambda\le\lambda^*$ then $u_{\mu}<u_\lambda$.

Hence we have proved (a), (b), (c).
\end{proof}

In   Example~\ref{elsobre01} we use the method used in the proof of the above theorem to characterize the minimal solutions.

  In the above theorem we have seen that under assumptions ~\eqref{H} we have $\lambda^*$ is bounded. Let us see now that if ~\eqref{H}-$(c)$ does not holds, then, as   in the local case, $\lambda^* = +\infty$.

\begin{proposition}\label{lambdainfty} Assume that $f$ satisfies hypotheses $(a)$ and $(b)$ of~\eqref{H} and
\begin{equation}\label{noc}
\hbox{for any $n\in \N$ there exists $s_n \geq 0$ such that $f(s_n) \leq \frac{s_n}{n}.$}
\end{equation}
Then, $\lambda^* = +\infty$.
\end{proposition}
\begin{proof} By Theorem \ref{EUp}, there exists a unique solution $\phi \geq 0$ of the problem
\begin{equation}\label{Dirichlet1NN}
\left\{\begin{array}{ll}
-\Delta_m u  = 1&\hbox{in } \Omega,\\[10pt]
u(x)=0&\hbox{on }\partial_m\Omega.
\end{array}\right.
\end{equation}
Moreover, $\phi \in L^\infty(\Omega,\nu)$, see~\eqref{cotalinf}. Set $M:= \Vert \phi \Vert_{L^\infty(\Omega,\nu)} < \infty.$

Fix $\lambda >0$. By assumption \eqref{noc}, there exists $s_n >0$ such that
\begin{equation}\label{e1No}
\frac{f(s_n)}{s_n} \leq \frac{1}{\lambda M}.
\end{equation}

Obviously,  $\underline{u} =0$ is a subsolution of problem~$(P(\lambda f))$.
We define the function
$$\overline{u}(x):= \left\{ \begin{array}{ll}  \displaystyle\frac{s_n \phi}{M} \quad &\hbox{if} \ \ x \in \Omega \\[10pt] 0 \quad &\hbox{if} \ \ x \in \partial_m\Omega. \end{array} \right.$$
 For $x \in \Omega$, we have
$$- \Delta_m \overline{u}(x) = -\displaystyle\frac{s_n}{M} \Delta_m \phi = \displaystyle\frac{s_n}{M} \geq \lambda f(s_n) \geq \lambda f(\overline{u}(x)).$$
Hence, $\overline{u}$ is a supersolution of problem~$(P(\lambda f))$. Then, since $\lambda f$ is non-decreasing, by Theorem \ref{Truncat}, we have that there exists a solution $w_\lambda$ of problem~$(P(\lambda f))$. Thus, since this holds for all $\lambda >0$, we have  $\lambda^* = +\infty$.
\end{proof}

\begin{remark}\label{remgi} {\rm  \noindent 1. Under the assumption~\eqref{H},
$$\lambda< \lambda_m(\Omega)\frac{1}{c_1}.$$
 Assume that  moreover $f$ is continuous  and $$\lim_{s\to +\infty}\frac{f(s)}{s}=+\infty.$$
 Then, for
$$
g(s):=\frac{s}{f(s)}, \quad s\geq 0,
$$
we have that   $g$ is continuous,  $0=g(0)\le g(s)\le \frac{1}{c_1}$ and $\lim_{s\to+\infty}g(s)=0;$ then, there exists $s_0>0$ such that
$$
g(s_0)=\max_{s\geq 0}g(s)
$$
so
 \begin{equation}\label{optimallambdamas}\lambda^* \leq  \lambda_m(\Omega)\frac{s_0}{f(s_0)}<\frac{s_0}{f(s_0)}.
\end{equation}
   We will see  that in some cases   $\lambda^* =  \lambda_m(\Omega)\frac{s_0}{f(s_0)}$.

\

\noindent  2.  In~\cite[Theorem~3.10]{MTtorsion} it is shown that, for general random walks,
$$\lambda_m(\Omega)=
1-\lim_n\sqrt[n]{\frac{g_{m,\Omega}(2n)}{g_{m,\Omega}(n)}}
$$
where
$$g_{m,\Omega}(0)=   \nu (\Omega)$$
and
$$g_{m,\Omega}(1)= \int_\Omega\int_\Omega dm_x (y)d\nu(x)  = L_m(\Omega, \Omega), $$
$$g_{m,\Omega}(2)= \int_\Omega\int_\Omega\int_\Omega dm_y(z)dm_x (y)d\nu(x),   $$
$$  \vdots$$
$$g_{m,\Omega}(n)= \int_{\hbox{\tiny$\underbrace{\Omega\times...\times\Omega}_n\times\Omega$}}dm_{x_n}(x_{n+1})\dots dm_{x_1}(x_2)d\nu(x_1)
$$
(see a probabilistic interpretation of these terms in such reference). In the case of weighted graphs, in~\cite[Appendix~A]{MTtorsion}, an iterative method is given to approximate $g_m(n)$. See~\cite{AGrigor} for an estimation from below of $\lambda_m(\Omega)$ via inradius.

\

\noindent
3.  Assume that $f\in \mathcal{C}^1([0,\infty[)$ is a strictly convex function satisfying~\eqref{C1}.  Observe that, in particular, \eqref{H}~is satisfied. Now, $g$, as defined in the previous point 1., belongs to $\mathcal{C}^1([0,\infty[)$, and for $s_0$ given  in the previous point,
$$
g^\prime (s_0)=\frac{f(s_0)-s_0\,f^\prime (s_0)}{f^2(s_0)}=0,
$$
which implies that $$f(s_0)-s_0\,f^\prime (s_0)=0.$$
Now, consider the function
$$
h(s)=f(s)-s\,f^\prime (s), \quad s\geq 0.
$$
  Observe that $h(0)=f(0)>0$, $h(s_0)=0$ and that
  $$
  h^\prime (s) =-s\,f^{\prime \prime}(s)\leq 0,
  $$
  since $f$ is  strictly convex, $h$ is decreasing and $h(s)> 0$ if $s<s_0$ and $h(s)<0$ if $s>s_0$. As a consequence, since $g^\prime(s)=\frac{h(s)}{f^2(s)}$, it follows that $g$ is increasing in $[0,s_0]$ and decreasing in $[s_0,\infty[$.
Therefore, the inverse of the function $g$ has two branches: the function $g^{-1}_1$ that is an increasing  continuous function in $\left[0,\frac{s_0}{f(s_0)}\right]$ with $$g^{-1}_1(0)=0,$$ and the function $g^{-1}_2$, that is a decreasing  continuous function in $\left]0,\frac{s_0}{f(s_0)}\right]$ with
$$\lim_{s\to 0^+}g_2^{-1}(s)=+\infty.$$
And $$g^{-1}_1\left(\frac{s_0}{f(s_0)}\right)=g^{-1}_2\left(\frac{s_0}{f(s_0)}\right).$$

Then, for every $0 < \lambda <  \lambda^*$, if $u$ is a solution to $(P(\lambda f))$, since $\lambda<g(u)$, we have
$$
g^{-1}_1(\lambda) \leq u \leq g^{-1}_2 (\lambda)  \quad \hbox{for all} \ \ \lambda \in [0, \lambda^*[.
$$
  In particular
\begin{equation}\label{boundWpar}
 0\le u_\lambda\le  g^{-1}_2 (\lambda^*)  \quad \hbox{for all} \ \ \lambda \in [0, \lambda^*[.
\end{equation}
\hfill$\blacksquare$
}\end{remark}

\begin{theorem}\label{extremal}
    Assume that $f\in \mathcal{C}^1([0,\infty[)$ is  strictly convex and satisfies~\eqref{C1}. Then,
there exists $u^*:=u_{\lambda^*}$ a bounded minimal solution to $(P(\lambda^* f))$. It is called the extremal solution.
\end{theorem}
\begin{proof}  Having in mind ~\eqref{boundWpar},  $u^*$ is obtained by taking limits to $u_\lambda$ as $\lambda\nearrow \lambda^*$. It is easy to see that it is minimal.
\end{proof}

 That the function $f$ must be strictly convex is a necessary condition. Note Example~\ref{aladm10} where the function is only convex and there is no solution for $\lambda=\lambda^*$.

\begin{corollary}\label{BoundI}  Assume that $f\in \mathcal{C}^1([0,\infty[)$ is a  strictly convex function satisfying~\eqref{C1}. Then, we can encapsulate any solution $u$ to problem~$(P(\lambda f))$, independently of the random walk space involved, in the following way:
\begin{equation}\label{boundW}
g^{-1}_1(\lambda) \leq u \leq g^{-1}_2 (\lambda)  \quad \hbox{for all} \ \ \lambda \in [0, \lambda^*].
\end{equation}
  In particular,
$$
  u_\lambda\le  g^{-1}_2 (\lambda^*)  \quad \hbox{for all} \ \ \lambda \in [0, \lambda^*].
$$
\end{corollary}

\begin{remark} \rm Let us point out that the above result shows a notable difference with the local case.  In the local case, for example, the bifurcation diagrams for $f(s)=e^s$ can not be encapsulated in such a similar way for dimensions $n\ge 3$ (see for instance \cite{D}).

  Moreover, for the nonlocal  fractional $s$-Laplacian $(- \Delta)^s$ in $\R^N$,  $0<s<1$, the extremal solution $u^*$ is bounded if $N < 10s$ (see \cite{R1}). Now, as consequence of Corollary \ref{BoundI},  for the nonlocal Laplacian $\Delta_{m^J}$ the extremal solution is bounded independently of the dimension $N$ and the kernel $J$. \hfill$\blacksquare$
\end{remark}

\begin{remark}
{\rm   Assume that $f\in \mathcal{C}^1([0,\infty[)$ is a  strictly convex function satisfying~\eqref{C1}. Then,
$$-\Delta_mu_\lambda=\lambda f(u_\lambda)\le \lambda f( g^{-1}_2 (\lambda^*))\le \lambda\frac{f( g^{-1}_2 (\lambda^*))}{\lambda_m(\Omega)\alpha_\Omega}\lambda_m(\Omega)\varphi_m$$
$$=-\Delta_m\left(\lambda\frac{f( g^{-1}_2 (\lambda^*))}{\lambda_m(\Omega)\alpha_\Omega}\varphi_m\right).$$
Therefore, by the maximum principle,
$$0\le u_\lambda \le \lambda\frac{f( g^{-1}_2 (\lambda^*))}{\lambda_m(\Omega)\alpha_\Omega}\varphi_m.$$
\hfill$\blacksquare$}\end{remark}

\subsection{Stability}

Assume that  $g \in C^1(\R)$   and consider the energy functional  introduced previously:
$$\mathcal{E}_{g}(u)= \frac{1}{4 } \int_{\Omega_m \times \Omega_m} \vert \nabla u(x,y) \vert^2 d(\nu\otimes m_x)(x,y) - \int_{\Omega} G(u(x)) d\nu(x),$$
where  $G' =g$.  It is easy to see that, for $u \in L^\infty(\Omega_m, \nu) \cap L_0^2(\Omega_m, \nu)$,
$$\langle D\mathcal{E}_{g}(u), v \rangle = \frac12\int_{\Omega_m \times \Omega_m} \nabla u (x,y) \cdot \nabla v(x,y)d(\nu\otimes m_x)(x,y) - \int_\Omega g(u(x)) v(x) d\nu(x),
$$
for $v \in  L^\infty(\Omega_m, \nu) \cap L_0^2(\Omega_m, \nu)$.
 That is,  if $$E_u(t,v) := \mathcal{E}_{g}(u + tv),$$ we have that
$$E^{\prime}_u(0,v) = \frac12\int_{\Omega_m \times \Omega_m} \nabla u (x,y) \cdot \nabla v(x,y)d(\nu\otimes m_x)(x,y) - \int_\Omega g(u(x)) v(x) d\nu(x).
$$

Similarly to~Lemma \ref{EnerSol}, we have the following result  for $L^\infty$-bounded  solutions of problem $(P(g))$.

 \begin{proposition}  Assume $g \in C^1(\R)$. If  $u$ is a  bounded  solution of problem $(P(g))$ then $E^{\prime}_u(0,v) =0$ for all direction $v \in L^\infty(\Omega_m, \nu) \cap L_0^2(\Omega_m, \nu)$.
 \end{proposition}

The following concept of stability is the same than the given for the local problem. See Dupaigne~\cite{D} for a large study of stability for   local problems, we use the same notation  than in such reference in this context.

\begin{definition} {\rm
 We say that a   solution $u$ of problem $(P(g))$ is {\it stable} if $E^{\prime \prime}_u(0,v) \geq 0$  for all direction  $v \in  L^\infty(\Omega_m, \nu) \cap  L_0^2(\Omega_m, \nu)$.}
\end{definition}

Let $g \in C^1(\R)$, and suppose that  $u$ is a  bounded solution  of problem $(P(g))$. We have that  $E^{\prime}_u(0,v) =0$ for all direction $v \in L^\infty(\Omega_m, \nu) \cap L_0^2(\Omega_m, \nu)$, and then
$$\frac{E^{\prime}_u(t,v)- E^{\prime}_u(0,v)}{t} =  \frac{\langle D\mathcal{E}_{g}(u+tv), v \rangle}{t} $$ $$\frac{1}{t} \, \frac{1}{2} \int_{\Omega_m \times \Omega_m} \nabla u (x,y) \cdot \nabla v(x,y)d(\nu\otimes m_x)(x,y) +\frac{1}{2} \int_{\Omega_m \times \Omega_m} \vert \nabla v (x,y) \vert^2 d(\nu\otimes m_x)(x,y)$$ $$-  \int_\Omega \frac{1}{t}(g((u+ tv)(x))- g(u)) v(x) d\nu(x).$$
Applying the Dominate Convergence Theorem, we get
$$E_u^{\prime\prime}(0,v) = \frac{1}{2} \int_{\Omega_m \times \Omega_m} \vert \nabla v (x,y) \vert^2 d(\nu\otimes m_x)(x,y)- \int_\Omega g'(u(x)) v^2(x) d\nu(x).$$
We define
$$
Q_u(v):= \displaystyle\frac{1}{2} \int_{\Omega_m \times \Omega_m} \vert \nabla v (x,y) \vert^2 d(\nu\otimes m_x)(x,y)-\displaystyle\int_\Omega g'(u(x)) v^2(x) d\nu(x).
$$
Then, we have the following result.

\begin{proposition}\label{equiv}  Let $g \in C^1(\R)$. Let $u$ be a   bounded solution  of problem $(P(g))$.
 Then,  $u$  is stable if and only if $Q_u(v)  \geq 0$
for all  $v \in L^\infty(\Omega_m, \nu) \cap L_0^2(\Omega_m, \nu)$.
\end{proposition}

\begin{remark}\label{positive test}{\rm
In  the above Proposition we can test only non-negative functions. Indeed, for $v \in L^\infty(\Omega_m, \nu) \cap L_0^2(\Omega_m, \nu)$,   $Q_u(v)\ge Q_u(v^+)+Q_u(v^-)$.
\hfill$\blacksquare$}\end{remark}

\begin{remark}\label{dom001}\rm    Assume that $f\in \mathcal{C}^1([0,\infty[)$ is a convex function satisfying~\eqref{C1}.  If $u$ is a solution to $(P(\lambda f))$, $0\le\lambda\le \lambda^*$, with $u\le s_0$, then it is  stable. In fact, let $v \in L^\infty(\Omega_m, \nu) \cap L_0^2(\Omega_m, \nu)$, then, by the 2-Poincar\'{e} type inequality~\eqref{Realyq},
$$\frac{1}{2} \int_{\Omega_m \times \Omega_m} \vert \nabla v (x,y) \vert^2 d(\nu\otimes m_x)(x,y)\ge \lambda_m(\Omega)\int v^2d\nu.$$
Now since $\lambda^* \leq  \lambda_m(\Omega)\frac{s_0}{f(s_0)}$ and $f'(s_0)=\frac{f(s_0)}{s_0}$, we get
$$\frac{1}{2} \int_{\Omega_m \times \Omega_m} \vert \nabla v (x,y) \vert^2 d(\nu\otimes m_x)(x,y)\ge \lambda_m(\Omega)\int v^2d\nu\ge\lambda f'(s_0)\int v^2d\nu,$$
and since $u\le s_0$,
$$\frac{1}{2} \int_{\Omega_m \times \Omega_m} \vert \nabla v (x,y) \vert^2 d(\nu\otimes m_x)(x,y) \ge\lambda f'(s_0) \int v^2d\nu\ge \lambda \int f'(u)v^2d\nu.$$
We will now improve this result. \hfill$\blacksquare$
\end{remark}

\begin{theorem}\label{megusta} Assume that  $f \in C^1([0,\infty[)$   satisfies hypotheses~\eqref{H}. For every $0 \leq \lambda < \lambda^*$, the minimal solution $u_\lambda$ of problem~$(P(\lambda f))$ is stable. Furthermore, if $u^*$ exists then it is also stable.
\end{theorem}

Therefore, by Theorem \ref{extremal} we have the following consequence,

\begin{corollary}
 Assume that $f\in \mathcal{C}^1([0,\infty[)$ is  strictly convex and satisfies~\eqref{C1}. Then,
the bounded minimal solution  $u_{\lambda}$ to $(P(\lambda f)))$ is stable for every $0\leq \lambda \leq \lambda^*$.
\end{corollary}

\begin{proof}
Suppose first that $0 \leq \lambda < \lambda^*$. In {\it Step 2.} of the proof of Theorem~\ref{existence}, we  construct a sequence $\{u_n\}_n$, with
$$u_n(x)<u_{n+1}(x)\quad\hbox{a.e.},$$ solutions of
$$\left\{\begin{array}{ll}
\displaystyle- \Delta_m u_n =\lambda \, f(u_{n-1}),\quad &\hbox{in} \ \Omega,
\\ \\
u_n=0,\quad &\hbox{on} \  \partial_m\Omega,
\end{array}\right.
$$
for $n\ge 1$,
where
$$\left\{\begin{array}{ll}
\displaystyle- \Delta_m u_0 =\lambda f(0),\quad &\hbox{in} \ \Omega,
\\ \\
u_0=0,\quad &\hbox{on} \  \partial_m\Omega,
\end{array}\right.
$$
 being their limit as $n\to+\infty$ the minimal solution $u:=u_\lambda$ of~$(P(\lambda f))$.  Observe that $u_n(x)<u(x)$ for $x\in\Omega$. Then,  for $v \in  L_0^2(\Omega_m, \nu)$, we have that,   for $\delta>0$ $$\frac{v^2}{u-u_n+ \delta }\in L_0^2(\Omega_m, \nu)$$
 and
\begin{equation}\label{alaalex01}\begin{array}{c}\displaystyle\frac12\int_{\Omega_m}\int_{\Omega_m}\nabla \left(u-u_n+ \delta \right)(x,y)\nabla\left(\frac{v^2}{u-u_n+ \delta }(x,y)\right)dm_x(y)d\nu(x)\\ \\
\displaystyle=\lambda\int_\Omega\frac{f(u)-f(u_{n-1})}{u-u_n+ \delta }v^2d\nu.
\end{array}
\end{equation}
Let us call $w_n=u_n-u+ \delta $,  then, applying~\eqref{Leib11dos},  the left hand side of~\eqref{alaalex01} can be written as
  \begin{equation}\label{alaalex03}\begin{array}{c}\displaystyle\frac12\int_{\Omega_m}\int_{\Omega_m}\nabla w_n(x,y)\nabla\left(\frac{v^2}{w_n }(x,y)\right)dm_x(y)d\nu(x) =\frac14\int_{\Omega_m}\int_{\Omega_m}\nabla w_n(x,y)\times\quad \quad\quad\\ \\
\displaystyle  \quad\times\frac{(v(y)+v(x))\nabla v(x,y)(w_n(y)+w_n(x))-(v(y)^2+v(x)^2)\nabla w_n(x,y)}{w_n(y)w_n(x)} dm_x(y)d\nu(x).
\end{array}
\end{equation}
Now, for non-negative numbers $a,b,c, d \in \R$, a simple calculation gives
$$\displaystyle
   (a-b) \frac{(c+d)(c-d)(a+b)-(c^2+d^2)(a-b)}{ab}
\displaystyle \le 2(c-d)^2.
$$
Hence, taking $a=w_n(y)$, $b=w_n(x)$, $c=v(y)$ and $d=v(x)$ in the last expression of~\eqref{alaalex03}, we obtain that
$$\begin{array}{c}\displaystyle
 \nabla w_n(x,y) \frac{(v(y)+v(x))\nabla v(x,y)(w_n(y)+w_n(x))-(v(y)^2+v(x)^2)\nabla w_n(x,y)}{w_n(y)w_n(x)}
\\ \\ \displaystyle  \le 2\left|\nabla v(x,y)\right|^2.
\end{array}$$
Therefore, from such inequality and~\eqref{alaalex03},
$$\displaystyle\frac12\int_{\Omega_m}\int_{\Omega_m}\nabla w_n(x,y)\nabla\left(\frac{v^2}{w_n }(x,y)\right)dm_x(y)d\nu(x) \leq  \frac{1}{2}\int_{\Omega_m}\int_{\Omega_m}\left|\nabla v(x,y)\right|^2dm_x(y)d\nu(x),$$
  and consequently, from~\eqref{alaalex01},  and using that $f(u_n)\ge f(u_{n-1})$,
$$
\begin{array}{c}\displaystyle\frac{1}{2}\int_{\Omega_m}\int_{\Omega_m}\left|\nabla v(x,y)\right|^2dm_x(y)d\nu(x)\ge \lambda\int_\Omega\frac{f(u)-f(u_{n-1})}{u-u_n+ \delta }v^2d\nu\\
\\ \displaystyle \ge   \lambda\int_\Omega\frac{f(u)-f(u_n)}{u-u_n+ \delta }v^2d\nu.
\end{array}
$$
Hence, applying Fatou's lemma twice, first taking limits as $\delta\to 0^+$ and afterwards as $n\to+\infty$,
$$\frac{1}{2}\int_{\Omega_m}\int_{\Omega_m}\left|\nabla v(x,y)\right|^2dm_x(y)d\nu(x)\ge \lambda\int_\Omega f'(u)v^2d\nu,$$
 and  by Proposition \ref{equiv} we have that  $u_\lambda$ is stable.

 Finally, if we assume the existence of $u^*$, by Proposition \ref{equiv} we have $Q_{u_\lambda}(v) \geq 0$ for all  $v \in L^\infty(\Omega_m, \nu) \cap L_0^2(\Omega_m, \nu)$ and  for $0 \leq \lambda <\lambda^*$. Then, by the Dominate Convergence Theorem, we have $Q_{u^*}(v) \geq 0$ for all  $v \in L^\infty(\Omega_m, \nu) \cap L_0^2(\Omega_m, \nu)$, and consequently $u^*$ is stable
\end{proof}

 The proof of the following result is similar to that for~\cite[Proposition 1.3.1]{D},  we give it for the sake of completeness.

\begin{proposition}\label{Uniq1} Let  $f \in C^1(\mathbb{R})$ be strictly convex. Then, there is a unique stable solution of problem $(P(\lambda f))$. Consequently,  if  $f \in C^1([0,+\infty))$ is strictly convex and satisfies~\eqref{H},    for  each $0\le \lambda  \leq \lambda^*$,  the minimal solution  $u_\lambda$ is the unique stable solution of problem $(P(\lambda f))$.
\end{proposition}

\begin{proof} Suppose there exist two stable solution $u_1$, $u_2$ of problem $(P(\lambda f))$. Then, $w:= u_2 - u_1$ is solution of
$$- \Delta_m w = \lambda (f(u_2) - f(u_1)) \quad \hbox{in} \ \Omega.$$
Multiplying the above equality by $w^+$ and integrating by parts, we obtain
$$\frac12\int_\Omega \vert \nabla w^+ \vert^2 d\nu   \le \lambda \int_\Omega (f(u_2) - f(u_1)) w^+ d\nu.$$
On the other hand, since $u_2$ is stable, we have
$$\frac12\int_\Omega \vert \nabla w^+ \vert^2 d\nu \geq  \lambda \int_\Omega f'(u_2)(w^+)^2 d\nu.$$
Hence
$$\int_\Omega \left(f(u_2) - f(u_1)  -  f'(u_2)w^+ \right) w^+ d\nu \geq 0.$$
Then,
$$\int_{\{u_2>u_1\}} \left(f(u_2) - f(u_1)  -  f'(u_2)(u_2-u_1)\right) (u_2-u_1) d\nu \geq 0.$$
 Then,
\begin{equation}\label{1323}\int_{\{u_2>u_1\}} \left(f(u_2) - f(u_1)  -  f'(u_2)(u_2-u_1)\right) (u_2-u_1) d\nu \geq 0.
\end{equation}
Now, since $f$ is strictly convex, we have $f(u_2) - f(u_1)  -  f'(u_2)(u_2-u_1)>0$ if $u_2>u_1$. Hence~\eqref{1323} implies that $\{u_2>u_1\}$ is $\nu$-null, so $u_2\le u_1$.
The  reverse inequality is obtained by exchanging the role of $u_1$ and $u_2$.
\end{proof}

 Following the example in~\cite[Proposition 1.3.3]{D},  we   see that the convexity assumption cannot be dropped in Proposition \ref{Uniq1}.

\begin{example}\rm Consider the Allen-Cahn nonlinearity $f(s)= s -s^3$ and consider the problem
 \begin{equation}\label{ACProb}\left\{\begin{array}{ll}
\displaystyle- \Delta_m u =\lambda \, (u -u^3),\quad &\hbox{in} \ \Omega,
\\ \\
u=0,\quad &\hbox{on} \  \partial_m\Omega.
\end{array}\right.
\end{equation}
Obviously, $\underline{u} =0$ is a solution of \eqref{ACProb}.
Now,
$$Q_u(v)= \displaystyle\frac{1}{2} \int_{\Omega_m \times \Omega_m} \vert \nabla v (x,y) \vert^2 d(\nu\otimes m_x)(x,y)- \lambda\displaystyle\int_\Omega (1 -3u(x)^2) v^2(x) d\nu(x).$$
Thus,
$$Q_{\underline{u}}(v) = \displaystyle\frac{1}{2} \int_{\Omega_m \times \Omega_m} \vert \nabla v (x,y) \vert^2 d(\nu\otimes m_x)(x,y) - \lambda\displaystyle\int_\Omega  v^2(x) d\nu(x) \ge 0$$ if and only if $$ 0\leq \lambda \leq \lambda_m(\Omega),$$
and then
$ {\underline{u}}$ is a stable solution of problem~\eqref{ACProb} for $0\leq \lambda \leq \lambda_m(\Omega).$
Moreover, the same proof given for \cite[Proposition 1.3.3]{D} shows that the energy functional
$$\mathcal{E}_{\lambda f}(w):= \frac{1}{4}  \int_{\Omega_m \times \Omega_m} \vert \nabla w(x,y) \vert^2 d(\nu\otimes m_x)(x,y) + \frac{\lambda}{4}\int_{\Omega} \left(w^2-1\right)^2 d\nu(x)$$
is strictly convex, for $0\leq \lambda \leq \lambda_m(\Omega),$ hence we have that $\underline{u} =0$ is the unique stable solution of problem \eqref{ACProb}.

Suppose $u$ is a solution of problem \eqref{ACProb} for $\lambda > \lambda_m(\Omega)$. Now,
$$Q_{u}(v) =  \displaystyle\frac{1}{2} \int_{\Omega_m \times \Omega_m} \vert \nabla v (x,y) \vert^2 d(\nu\otimes m_x)(x,y)  - \lambda\displaystyle\int_\Omega \left(1-3u(x)^2 \right) v^2(x) d\nu(x) $$ $$=  \displaystyle\frac{1}{2} \int_{\Omega_m \times \Omega_m} \vert \nabla v (x,y) \vert^2 d(\nu\otimes m_x)(x,y)  - \lambda\displaystyle\int_\Omega v^2(x) d\nu(x) + 3\lambda \displaystyle\int_\Omega  u(x)^2 v^2(x) d\nu(x) \geq 0,$$
by the $2$-Poincar\'{e} inequality. Thus,
$ u \ \hbox{is stable } \ \forall \, \lambda > \lambda_m(\Omega).$
And, therefore $u \not= \underline{u},$ i.e.,
$$u\not= 0.$$ Then, since $-u$ is also a solution, it will be stable and we have that there will exist at least two nontrivial stable solutions.

  Let us  show that there exists examples for which there is a solution of problem \eqref{ACProb} for some $\lambda> \lambda_m(\Omega)$.

We define the function
$$\overline{u}(x):= \left\{ \begin{array}{ll}  1 \quad &\hbox{if} \ \ x \in \Omega \\[10pt] 0 \quad &\hbox{if} \ \ x \in \partial_m\Omega. \end{array} \right.$$
 For $x \in \Omega$, we have
$$- \Delta_m \overline{u}(x) = 1- \int_\Omega  \overline{u}(y) dm_x(y) = 1- m_x(\Omega) \geq 0 =\lambda \, ( \overline{u}(x) - \overline{u}(x)^3).$$
Hence, $\overline{u}$ is a supersolution of problem \eqref{ACProb}.

By Theorem \ref{Truncat}, if
$$\h_\lambda(x,r):= \left\{ \begin{array}{lll} \lambda f(\underline{u}(x)) \quad &\hbox{if} \ \ r < \underline{u}(x), \\[10pt] \lambda f(r) \quad &\hbox{if} \ \ \underline{u}(x) \leq r \leq \overline{u}(x) \\[10pt]  \lambda f(\overline{u}(x)) \quad &\hbox{if} \ \ r >\overline{u}(x), \end{array} \right.$$
there is a solution $w_\lambda$ of problem  $(P(\h_\lambda))$. In this case, for $x \in \Omega$, we have
$$\h_\lambda(x,r):= \left\{ \begin{array}{lll}0 \quad &\hbox{if} \ \ r < 0, \\[10pt]  \lambda(r-r^3) \quad &\hbox{if} \ \ 0 \leq r \leq 1 \\[10pt]  0 \quad &\hbox{if} \ \ r > 1. \end{array} \right.$$
Now,
$$\left\{\begin{array}{ll} - \Delta_m w_\lambda (x) = \h_\lambda(x,w_\lambda(x)) \geq 0\,\quad &\hbox{if} \ x \in \Omega,
\\ \\
w_\lambda=0,\quad &\hbox{on} \  \partial_m\Omega,\end{array} \right.$$
Then, by the Maximum Principle, $$w_\lambda \geq 0.$$
On the other hand, if $v:= \overline{u} - w_\lambda$, for $x \in \Omega$, we have
$$- \Delta_m v(x) = 1 - m_x(\Omega) + \Delta_m w_\lambda(x) = 1 - m_x(\Omega)- \h_\lambda (x,w_\lambda(x)).
$$
Hence
$$- \Delta_m v(x) \geq 0 \iff \h_\lambda (x,w_\lambda(x)) \leq 1 - m_x(\Omega).$$
Then, since
$$\h_\lambda (x,w_\lambda(x)) \leq \frac{2}{3\sqrt{3}} \lambda,$$
we have that
\begin{equation}\label{condOmega1}- \Delta_m v(x) \geq 0 \quad \hbox{if} \ \ m_x(\Omega) \leq 1- \frac{2 \lambda}{3\sqrt{3}} \ \ \forall x \in \Omega.
\end{equation}
Then, under the condition \eqref{condOmega1}, applying the Maximum Principle, we get $w_\lambda \leq \overline{u}$. Consequently, under the condition \eqref{condOmega1}, we have $\underline{u} \leq w_\lambda \leq \overline{u}$, and therefore $w_\lambda$ is a solution of  of problem \eqref{ACProb}.

Finally, let see   examples of random walk spaces that satisfies the condition \eqref{condOmega1} for some $\lambda> \lambda_m(\Omega)$.
 Let $G=(V,E)$ be the  graph $V= \{1,2,3,4\}$, with weights
 $$w_{1,2}=w_{34} = b>0,\  w_{23}=a>0,$$ and $w_{i,j} = 0$ otherwise (see Figure~\ref{fig03}), and let $\Omega:= \{2,3\}$. We have, $\lambda_m(\Omega) = \frac{b}{a+b}$ and
 $m_2(\Omega) =  m_2(\Omega) = \frac{a}{a+b}$. Then, for $\lambda > \frac{b}{a+b}$, we need to find $a,b$ satisfying
$$\frac{a}{a+b} \leq 1- \frac{2 \lambda}{3\sqrt{3}},$$
which is equivalent to
$$ \lambda \le \frac{3\sqrt3}{2}\frac{b}{a+b}.$$
Then we need
$$\frac{b}{a+b} < \lambda \leq   \frac{3\sqrt{3}}{2}  \frac{b}{a+b},$$
and this is true for any $a,b>0$, and many $\lambda > \lambda_m(\Omega)=\frac{b}{a+b}$. \hfill$\blacksquare$
\end{example}

\begin{remark} {\rm  Let $g \in C^1(\R)$.
Observe that   $u$ is a  bounded stable solution of  problem $(P(g))$ if and only if
$$\mu_1(u) : =\inf_{v \in L_0^2(\Omega_m, \nu),   \Vert v \Vert_2 =1} \displaystyle\left(\frac{1}{2} \int_{\Omega_m \times \Omega_m} \vert \nabla v (x,y) \vert^2 d(\nu\otimes m_x)(x,y)- \displaystyle\int_\Omega g'(u) v^2 d\nu\right) \geq 0.$$

Moreover, by Remark \ref{positive test}, we can take this infimum for $v\geq 0$.

Let $u$ be a  bounded solution of  problem $(P(g))$ and  $L_u$  the linearized operator around $u$ defined as
$$L_u(v) := -\Delta_m v -  g'(u)v, \quad v \in L_0^2(\Omega_m, \nu),$$
and consider the eigenvalue problem
\begin{equation}\label{eigenproblemLinear}\left\{\begin{array}{l}
\displaystyle L_u v(x) =\mu v(x),\quad \nu\hbox{-a.e.} \ x\in\Omega,
\\ \\
v(x)=0,\quad x\in\partial_m\Omega.
\end{array}\right.
\end{equation}
If there exists $ \varphi\in L_0^2(\Omega_m, \nu)$, $ \varphi \not=0$, solution of \eqref{eigenproblemLinear}, we say that $\varphi$ is an eigenfunction and $\mu$ an eigenvalue of problem \eqref{eigenproblemLinear}. We will denote by $\lambda_1(- \Delta_m -g'(u))$  and by $\varphi_1$  the smallest eigenvalue  and eigenfunction  of  problem \eqref{eigenproblemLinear}.

In the case that  $\lambda_1(- \Delta_m -g'(u))$ exists, then $\lambda_1(- \Delta_m -g'(u)) = \mu_1(u)$, and we have
\begin{equation}\label{eigent1}
\hbox{$u$ is stable} \iff \lambda_1(- \Delta_m -g'(u)) \geq 0.
\end{equation}

   We have that $\lambda_1(- \Delta_m -g'(u))$ exists for finite weighted graphs, but, in general we do not know if it exists.

 Let us see that, under certain assumptions, $\lambda_1(- \Delta_m -g'(u))$ exists.  Since $L_u$ is a self-adjoint operator, by \cite[Proposition 6.9]{BrezisAF}, we have
 $$m:= \inf_{v \in L_0^2(\Omega_m, \nu),  \Vert v \Vert_2 =1} \int_{\Omega_m} L_u(v) v d\nu \in \sigma(L_u),$$
and  by integration by parts, we have $m = \mu_1(u)$, thus $$\mu_1(u) \in \sigma(L_u).$$
Now, we have $L_u(v) = Kv - v - g'(u) v$, with
$$Kv(x):= \int_{\Omega_m} v(y) dm_x(y).$$
So, assuming that
\\
1.   $g'(u)\ge c>0$  (or  $g'(u)\le c<0$), which implies that  the operator $v \mapsto g'(u) v$ is invertible,
\\
2. $K$ is a compact operator in  $L_0^2(\Omega_m, \nu)$,
\\
3. and $\mu_1(u) \not=0$,
\\
then,  by  Fredholm’s alternative, we have $\mu_1(u) \in \sigma_p(L_u)$, and consequently $$\exists \,  \lambda_1(- \Delta_m -g'(u))=\mu_1(u).$$
Remark that the above assumption 2. holds true for  finite weighted graphs and  for the random walk space $m^J$ given in Example \ref{example.nonlocalJ}.
\hfill$\blacksquare$
}
\end{remark}

 As in the local case (see for instance \cite{CEP}) we have the following result. The proof is similar, we include it for the sake of completeness.

\begin{proposition}\label{good1}  Assume that  $f \in C^1([0,\infty[)$ is a convex  or concave function  satisfying~\eqref{H}.  Let $0<\lambda < \lambda^*$. If $u$ is a stable bounded solution of problem $(P(\lambda f))$  and $\lambda_1(- \Delta_m -\lambda f'(u))$ exists,  then
$$\lambda_1(- \Delta_m -\lambda f'(u)) > 0.$$
\end{proposition}

\begin{proof}   Assume by contradiction that there exists $0<\lambda < \lambda^*$ and $u$   a bounded  stable solution of problem $(P(\lambda f))$ such that
$$\lambda_1(- \Delta_m -\lambda f'(u)) = 0,$$
(by \eqref{eigent1}, we have $\lambda_1(- \Delta_m -\lambda f'(u)) \geq  0$).
Then,   there exists $\varphi_1 \not=0$, such that
\begin{equation}\label{E1good1}
- \Delta_m \varphi_1  = \lambda f'(u) \varphi_1.
\end{equation}
We have $$Q_{u}(\varphi_1) = \inf_{v \in L_0^2(\Omega_m, \nu),   \Vert v \Vert_2 =1} Q_{u}(v).$$ And, since $(a-b)^2 \geq (\vert a \vert - \vert b \vert)^2$, we can assume that  $\varphi_1  \geq 0$.

On the other hand, we have
\begin{equation}\label{E2good1}
- \Delta_m u = \lambda f(u).
\end{equation}
Multiplying \eqref{E1good1} by $u$, \eqref{E2good1} by $\varphi_1$,  integrating by parts, and substracting the obtained equations, we get
\begin{equation}\label{E3good1}
\int_\Omega \left( \lambda f(u) - \lambda \, u f'(u)  \right) \varphi_1 \, d\nu =0.
\end{equation}
Similarly, if $\mu \in (0,\lambda^*)$, and $w$ is a bounded solution of problem  $(P(\mu f))$, we get
\begin{equation}\label{E4good1}
\int_\Omega \left( \mu f(w) - \lambda \, w f'(u)  \right) \varphi_1 \, d\nu =0.
\end{equation}
Subtracting \eqref{E3good1} from \eqref{E4good1}, we deduce that
\begin{equation}\label{E5good1}
\lambda \int_\Omega \left(f(u) - f(w) + (w - u)f'(u)\right) \varphi_1 \, d\nu = (\mu- \lambda) \int_\Omega f(w)  \varphi_1 \, d\nu.
\end{equation}
  If   $f$ is convex, we have that $f(u)+f'(u)(w-u)\le f(w)$, hence the left-hand side of \eqref{E5good1} is  non-positive and we get a contradiction with the sign of the right-hand side by choosing $\mu > \lambda$.  In an analogous way, if $f$ is concave, a contradiction is reached by taking $\mu < \lambda$.
\end{proof}

  Next, we prove that the branch of minimal solutions of $(P(\lambda f))$ is a continuum. To prove that we need the following result.
\begin{lemma}\label{Existence2}  Let $f \in C^1([0,\infty[)$  and $0<\lambda < \lambda^*$.
Assume that there exists $\lambda_1(- \Delta_m -\lambda f'(u_\lambda)) > 0$. Then,  for every $\varphi \in L^2(\Omega, \nu)$,
there is a solution  of problem
\begin{equation}\label{Dirichlet2}
\left\{\begin{array}{ll}
-\Delta_m u -  \lambda f'(u_\lambda)u  =\varphi&\hbox{in } \Omega,\\[10pt]
u(x)=0&\hbox{on }\partial_m\Omega.
\end{array}\right.
\end{equation}
\end{lemma}
\begin{proof}
Set the energy functional associated with problem \eqref{Dirichlet2}
$$
\tilde{ \mathcal{E}}_{\varphi}(u):= \frac{1}{4 } \int_{\Omega_m \times \Omega_m} \vert \nabla u(x,y) \vert^2 d(\nu\otimes m_x)(x,y)-\frac{1}{2} \int_\Omega \lambda f'(u_\lambda) u^2 d\nu-\int_\Omega \varphi \, u\,d\nu ,
$$

Taking in account that $\lambda_1:=\lambda_1(- \Delta_m -\lambda f'(u_\lambda)) > 0$, then
  \begin{equation}\label{a1}
   0< \lambda_1\, \int_\Omega \left|v(x) \right|^2 \, {d\nu(x)} \leq   \displaystyle \frac{1}{2} \int_{\Omega_m \times \Omega_m} \vert \nabla v (x,y) \vert^2 d(\nu\otimes m_x)(x,y)- \displaystyle\int_\Omega \lambda f'(u_\lambda) v^2 d\nu,
        \end{equation}
         for every $v \in L_0^2(\Omega_m, \nu)$.

Set $$\theta := \inf_{h \in L_0^2(\Omega_m, \nu) \setminus \{ 0 \}} \tilde{ \mathcal{E}}_{\varphi}(h),$$ and let $\{ u_n \}$ be a minimizing sequence, hence
$$\theta = \lim_{n \to \infty}\tilde{ \mathcal{E}}_{\varphi}(u_n).$$
 By \eqref{a1}, we have
 $$ 0< \lambda_1\, \int_\Omega \left|u_n(x) \right|^2 \, {d\nu(x)} \leq  2\tilde{ \mathcal{E}}_{\varphi}(u_n)+2\int_\Omega \varphi \, u_n d\nu.$$
 Then, by Young's inequality, we have $\{ u_n \ : \ n \in \N \}$ is bounded in $L_0^2(\Omega_m, \nu)$. Hence, up to a subsequence, we have
$$u_n \rightharpoonup u \ \ \ \hbox{in }   L_0^2(\Omega_m, \nu).$$
Furthermore, using the weak lower semicontinuity of the functional $\tilde{ \mathcal{E}}_{\varphi}$, we get
$$\tilde{ \mathcal{E}}_{\varphi}(u) = \theta,$$
so $u$ is a critical point of $\tilde{ \mathcal{E}}_{\varphi}$, which implies that $u$ is solution  of problem
\eqref{Dirichlet2}.
\end{proof}

\begin{theorem}\label{minimal-continua}
Assume that $f \in C^1([0,\infty[)$ is convex and satisfies hypotheses~\eqref{H}.   Let $\{u_\lambda \}_{0\leq \lambda<\lambda^*}$ be the branch of minimal solutions of $(P(\lambda f))$.  Assume also that, for each $0<\lambda<\lambda^*$,  $\lambda_1(- \Delta_m -\lambda f'(u))$ exists. Then:

\item  {\rm (a)} The mapping $\lambda \to \|u_\lambda\|_2$ is  continuous in $[0,\lambda^*[$.

\item  {\rm (b)}  In  finite weighted graphs, the mapping $\lambda \to \|u_\lambda\|_\infty$ is  continuous in $[0,\lambda^*[$.

\end{theorem}

\begin{proof}
 Firstly, we prove that $u_{\lambda_n}$ converges pointwise to $u_\lambda$, as $\lambda_n \to \lambda$.

  For this purpose, we start by setting  $0<\lambda<\lambda^*$. Let $0<\{\lambda_n\}_n<\lambda$ be an increasing sequence such that $\lambda_n\nearrow \lambda$ as $n\to \infty.$ Thus, $u_{\lambda_n}\leq u_\lambda$ and by Dominate Convergence Theorem $u_{\lambda_n}\to v$, a solution to problem $(P(\lambda f))$. Besides $v\leq u_\lambda$, but $u_\lambda$ is the minimal solution, so $v=u_\lambda$ and
  \begin{equation}\label{increasing}
  u_{\lambda_n}(x) \to u_\lambda, \quad \, \lambda_n\nearrow \lambda.
  \end{equation}

    In order to prove the other alternative, we continue in an analogous way as before. Let $ 0\le \lambda<\{\lambda_n\}_n<\lambda^*$ be a decreasing sequence such that $\lambda_n\searrow \lambda$ as $n\to \infty.$ Thus, $u_{\lambda_n}\geq u_\lambda$ and by Dominate Convergence Theorem $u_{\lambda_n}\to w$, a solution to problem $(P(\lambda f))$. Besides $w\geq u_\lambda$.

     Now, since $u_\lambda$ is minimal and $f$ is convex, then $u_\lambda$ is stable (Theorem \ref{megusta}). Even more, by Theorem \ref{good1}, $\lambda_1:=\lambda_1(- \Delta_m -\lambda f'(u_\lambda)) > 0$. Then, by then Lemma~\ref{Existence2},
there is a solution $\phi \in L_0^\infty(\Omega_m, \nu)$ of problem
$$
\left\{\begin{array}{ll}
-\Delta_m u - f'(u_\lambda)u  =\1_\Omega&\hbox{in } \Omega,\\[10pt]
u(x)=0&\hbox{on }\partial_m\Omega.
\end{array}\right.
$$

 Arguing as in \cite[Theorem 1.1]{CEP}, consider $v_\delta:=u_\lambda+\delta \phi$, for $\delta>0$. Clearly, $v_\delta \to u_\lambda$ as $\delta \to 0$.
Now, for a fixed $\theta>0$, the following holds:
\begin{align}
-\Delta_m v_\delta -(\lambda +\theta)f(v_\delta) & = \lambda f(u_\lambda)+ \delta+\delta \lambda f'(u_\lambda)\phi-\lambda f(v_\delta)-\theta f(v_\delta)
\\
& =\delta-\theta f(v_\delta)-\lambda \left(f(v_\delta)-f(u_\lambda)-\delta f'(u_\lambda)\phi \right)
\\
& = \delta-\theta f(v_\delta)-\lambda \left(f(v_\delta)-f(u_\lambda)-(v_\delta-u_\lambda)f'(u_\lambda)  \right).
\end{align}
Then, for $\tau_\delta \in (u_\lambda, u_\lambda+\delta \phi)$ such that $f(v_\delta)-f(u_\lambda)=f'(\tau_\delta)(v_\delta-u_\lambda)$,
\begin{align}\label{eq123}
-\Delta_m v_\delta -(\lambda +\theta)f(v_\delta) & =  \delta-\theta f(v_\delta)-\lambda \delta \phi \left(f'(\tau_\delta)-f'(u_\lambda) \right).
\end{align}
Since $f$ is $C^1$  and $u_\lambda$ is $L^\infty$--bounded, we take $\delta$, small enough, such that
$$
f'(\tau_\delta)-f'(u_\lambda)\leq \displaystyle \frac{1}{2\lambda \|\phi\|_\infty}.
$$
Take also
 $$
 \theta \leq \displaystyle \frac{\delta}{2f(\|u_\lambda\|_\infty+\delta \|\phi\|_\infty)}.
 $$
Then, from~\eqref{eq123},
\begin{align}
-\Delta_m v_\delta -(\lambda +\theta)f(v_\delta) & \geq \delta-\theta  f(v_\delta)-\lambda \delta \|\phi\|_\infty  \left(f'(\tau_\delta)-f'(u_\lambda) \right)
\\
& \geq \displaystyle \frac{\delta}{2}-\theta  f(v_\delta)
\\
& \geq \displaystyle \frac{\delta}{2}-\theta  f(\|u_\lambda\|_\infty + \delta \|\phi\|_\infty)
\\
&\geq 0.
\end{align}
 Thus, for $\theta$ small enough, $v_\delta$ is a supersolution of problem $(P((\lambda+\theta)f))$, so $v_\delta \geq u_{\lambda+\theta}$ and then
 $$
 w\leq v_\delta
 $$
 since $u_{\lambda_n}\to w$, as $\lambda_n \searrow \lambda$. Now, taking limit $\delta \to 0$, we get that $w\leq u_\lambda$. Hence $w=u_\lambda$ and  $u_{\lambda_n}(x) \to u_\lambda, \quad \, \lambda_n\searrow \lambda$. Therefore, together with \eqref{increasing}, we obtain the pointwise convergence:
 \begin{equation}\label{pointwise}
 u_{\lambda_n}(x)\to u_\lambda(x)\, \it{as} \lambda_n \to \lambda,\, \, \it{for}\, \it{a.e.} \, x\in \Omega.
  \end{equation}
  Moreover, $u_{\lambda_n}$ is uniformly bounded (because is increasing and bounded for all $n$).

\item  {\rm (a)} Can be obtained directly from \eqref{pointwise} and Dominate Convergence Theorem.
\item  {\rm (b)} Since the convergence $u_{\lambda_n}\to u_\lambda$ is uniform. It is straightforward that  $\lambda \to \|u_\lambda\|_\infty$ is continuous.
\end{proof}

\begin{remark}{\rm
As can be seen in Example \ref{exnocont} (see also Figure \ref{no_continuo_++}), the convexity of function $f$ is a necessary hypothesis in the above result. \hfill$\blacksquare$
}
\end{remark}

\section{The exponential function and   power-like functions. Examples}

\subsection{The Gelfand problem}
In the case that $f(s)=e^s$ we have that
$$\lambda^*\le\frac{\lambda_m(\Omega)}{e}<\frac{1}{e},$$
 and~\eqref{boundW} says that,  if $u$ is a solution to $(P(\lambda e^s))$,
\begin{equation}\label{boundWexp}
-W_0(-\lambda) \leq u \leq -W_{-1}(-\lambda).
\end{equation}
where   $W_0$ and $W_{-1}$ are  the principal branch and the negative branch, respectively, of the Lambert $W$ function, see Figure~\ref{figLambert}.
\begin{figure}[h]
\includegraphics[scale=0.5]{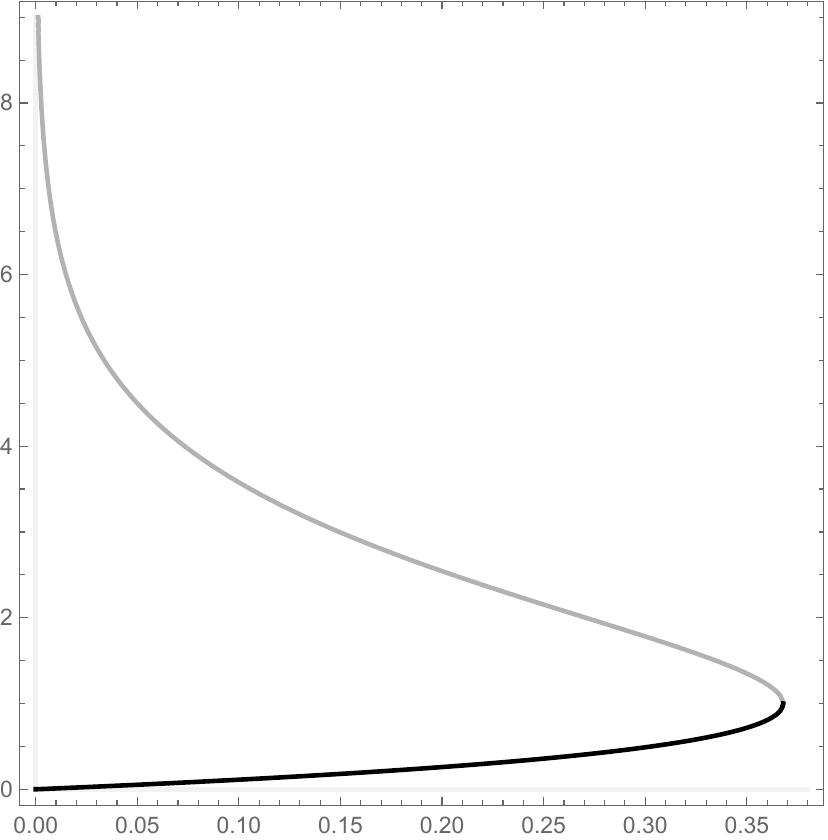}
\caption{$-W_0(-\lambda)$ and $-W_{-1}(-\lambda)$.}
\label{figLambert}
\end{figure}

\begin{remark} \rm
  Observe  that   if $u$ is a solution to $(P(\lambda e^s))$,  then
$$-u(x)e^{-u(x)}=  -\lambda  - e^{-u(x)}  \int_{\Omega} u(y)\, dm_x(y).$$
Therefore, for (a.e.) $x\in\Omega$, either
$$u(x)=-W_0\left(-\lambda  - e^{-u(x)} \int_{\Omega} u(y)\, dm_x(y)\right),$$
or
$$u(x)=-W_{-1}\left(-\lambda  - e^{-u(x)}  \int_{\Omega} u(y)\, dm_x(y)\right).$$
 Being both given implicitly.
In Example~\ref{example4pt} (also in Example~\ref{kn}) we see  that {\it the  branch of minimal solutions} satisfy
$$u_\lambda(x)=-W_0\left(-\lambda  - e^{-u(x)} \, \int_{\Omega} u_\lambda(y)\, dm_x(y)\right)\quad  \hbox{for all } x\in\Omega,$$
But this is not the general rule as seen in other examples.
\hfill$\blacksquare$\end{remark}

 The next examples are given in  quite simple  weighted graphs but they are illustrative of the many situations that can occur for the solutions of the Gelfand-type problems on weighted graphs. The   function
$$h(x)=x-\lambda e^{x},\quad\hbox{for $0<\lambda< \frac{1}{e}$,}
$$
  is important in these examples. It is a convex function with two zeros,  $-W_0(-\lambda)$ and $-W_{-1}(-\lambda)$, and a maximum at $\ln\frac{1}{\lambda}>0$ with value $\ln\frac{1}{\lambda}-1>0$. Its shape can be seen at Figure~\ref{landalargo}.

\begin{example}\label{example4pt}{\rm   Let $G=(V,E)$ be the  graph $V= \{1,2,3,4\}$, with weights $$w_{12}=w_{34}=w_{23}=1,$$ and $w_{i,j} = 0$ otherwise. This is the weighted linear graph given in Figure~\ref{fig02}.
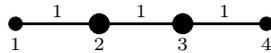
\begin{figure}[ht]
\centering
\begin{tikzpicture}[line cap=round,line join=round,>=triangle 45,x=1.1cm,y=1.1cm]
\draw [line width=1pt] (-2,0)-- (1,0);
\begin{scriptsize}
\draw [fill=black] (-2,0) circle (2.5pt);
\draw[color=black] (-2,-0.25) node {$1$};
\draw [fill=black] (-1,0) circle (3.75pt);
\draw[color=black] (-1,-0.25) node {$2$};
\draw[color=black] (-1.5,0.15) node {$1$};
\draw [fill=black] (0,0) circle (3.75pt);
\draw[color=black] (0,-0.25) node {$3$};
\draw[color=black] (-0.5,0.15) node {$1$};
\draw [fill=black] (1,0) circle (2.5pt);
\draw[color=black] (1,-0.25) node {$4$};
\draw[color=black] (0.5,0.15) node {$1$};
\end{scriptsize}
\end{tikzpicture}
\caption{Linear weighted graph in Example~\ref{example4pt}.}\label{fig02}
\end{figure}
 Set $\Omega:= \{2,3\}$ and consider the Gelfand problem $(P(\lambda e^s))$
\begin{equation}\label{pb00001}\left\{\begin{array}{ll}
\displaystyle- \Delta_{m^G} u =\lambda \, e^{u},\quad &\hbox{in} \ \Omega,
\\ \\
u=0,\quad &\hbox{on} \  \partial_{m^G}\Omega.
\end{array}\right.
\end{equation}
For this random walk space, it is easy to see that $\lambda_m(\Omega)=\frac12$. We have that $u$ is a solution of  \eqref{pb00001}  if and only if verifies
$$\left\{\begin{array}{ll}
\displaystyle- \Delta_{m^G} u(2) =\lambda \, e^{u(2)},\\
\\ - \Delta_{m^G} u(3) =\lambda \, e^{u(3)},
\\ \\
u(1) = u(4) = 0,
\end{array}\right.
$$
which is equivalent to $u(1) = u(4) = 0$ and
\begin{equation}\label{pb000012}\left\{\begin{array}{ll}
u(2)  -\frac{1}{2}u(3) = \lambda  \, e^{u(2)},
\\ \\
u(3)  -\frac{1}{2}u(2) = \lambda  \, e^{u(3)}.
\end{array}\right.
\end{equation}
Let us call $x=u(2)$ and $y=u(3)$. Then~\eqref{pb000012} is equal to
\begin{equation}\label{problem3Examxy}\left\{\begin{array}{ll}
y=2(x-\lambda e^{x}),
\\ \\
x=2(y-\lambda e^y).
\end{array}\right.
\end{equation}
Now, the graphs $y=2(x-\lambda e^{x})$ and $x=2(y-\lambda e^y)$, for $\lambda$  for large, do not intersect; for $\lambda=\lambda^*=\frac{1}{2e}= \lambda_m(\Omega)\frac{1}{e}$, they intersect in $(1,1)$; there exists $\hat\lambda\approx 0.076$ such that for $\lambda\in [\hat\lambda, \lambda^*)$, the graphs intersect in two points in the line $y=x$, $(u_\lambda(2),u_\lambda(3))$ and $(u^\lambda(2),u^\lambda(3))$; and for  $\lambda\in (0, \hat\lambda)$, they intersect in four points, that are $(u_\lambda(2),u_\lambda(3))$ and $(u^\lambda(2),u^\lambda(3))$ and
$(u^{3,\lambda}(2),u^{3,\lambda}(3)),\ (u^{4,\lambda}(2),u^{4,\lambda}(3)),$
with
$$u^{3,\lambda}(2)=u^{4,\lambda}(3)<u^\lambda(2)=u^\lambda(3)<u^{3,\lambda}(3)=u^{4,\lambda}(2),$$
see Figure~\ref{landalargo}.
\begin{figure}[ht]
  \centering
  \begin{subfigure}{0.25\textwidth}
    \centering
    \includegraphics[width=\linewidth]{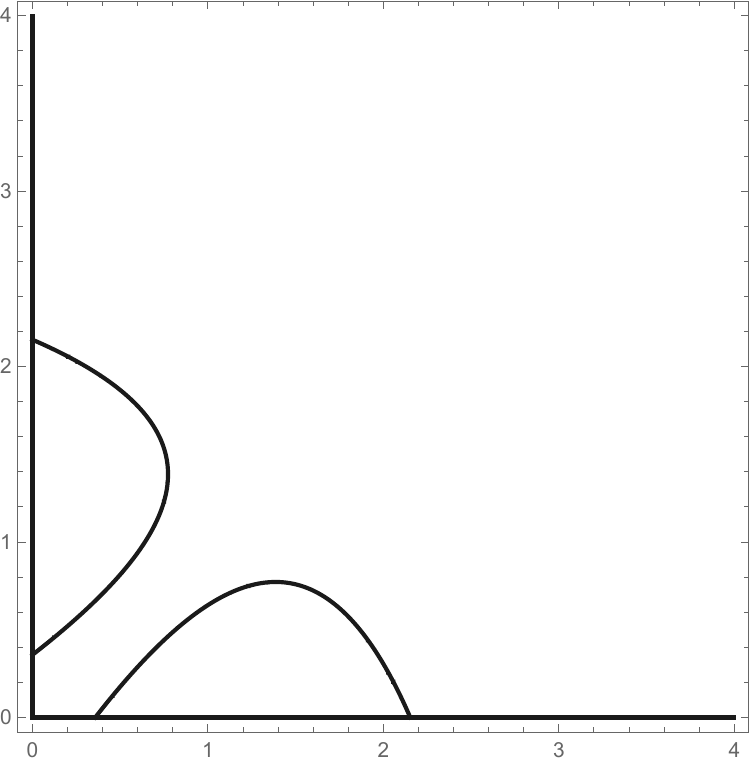}
    \caption{$\frac{1}{e}>\lambda>\lambda^*$.}
    %\label{fig:subfig1001}
  \end{subfigure}
  \hspace{1cm}
   \begin{subfigure}{0.25\textwidth}
    \centering
    \includegraphics[width=\linewidth]{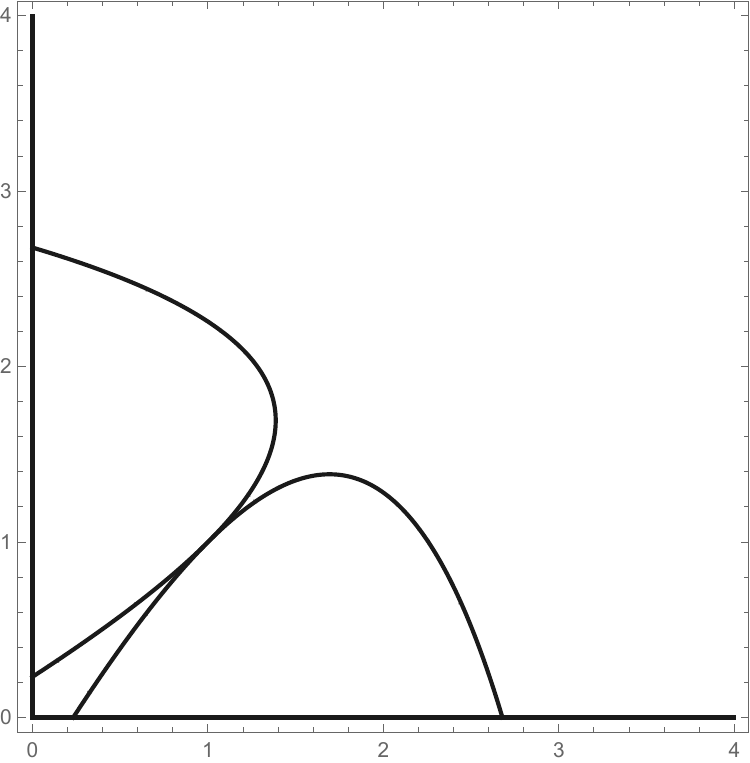}
    \caption{$\lambda=\lambda^*=\frac{1}{2e}$.}
    %\label{fig:subfig2002}
  \end{subfigure}

 \vspace{0.5cm}

  \begin{subfigure}{0.25\textwidth}
    \centering
    \includegraphics[width=\linewidth]{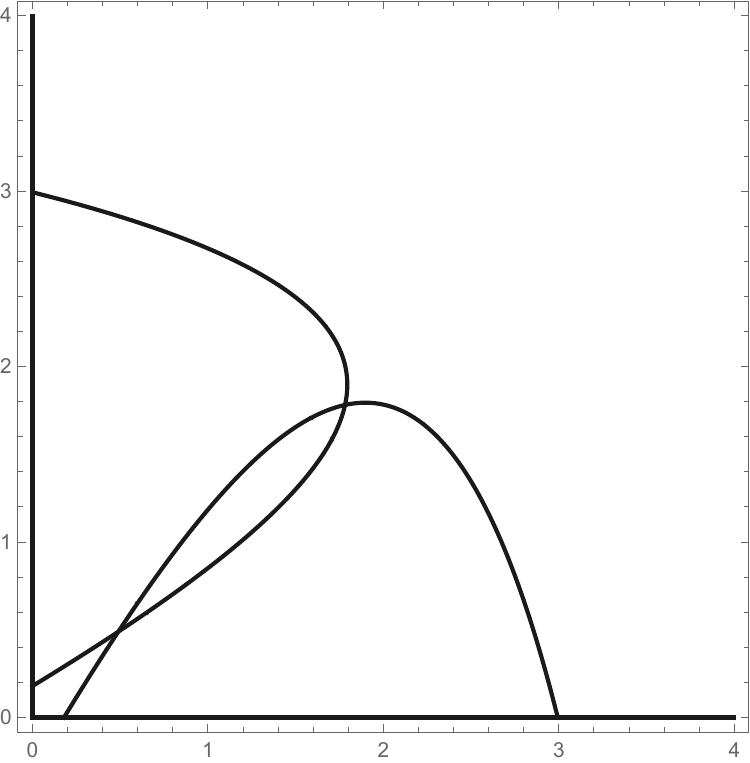}
    \caption{$\hat\lambda\le \lambda<\lambda^*$.}
    %\label{fig:subfig3003}
  \end{subfigure}
 \hspace{1cm}
   \begin{subfigure}{0.25\textwidth}
    \centering
    \includegraphics[width=\linewidth]{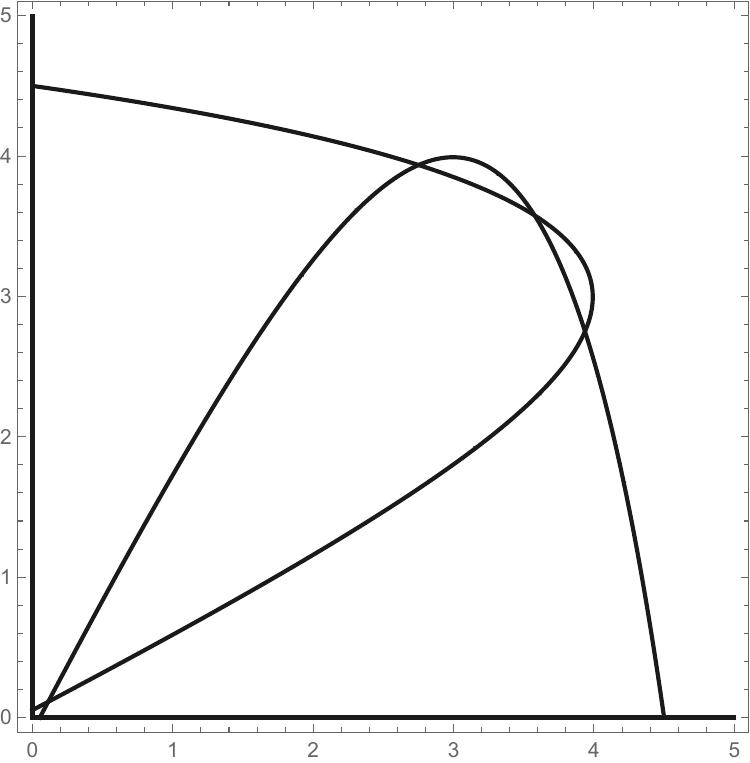}
    \caption{$0<\lambda<\hat\lambda$.}
    %\label{fig:subfig4004}
  \end{subfigure}

  \caption{Solutions of~\eqref{problem3Examxy}. In dark--gray the graph $x=2(y-\lambda e^y)$, in light--gray $y=2(x-\lambda e^{x})$.}
  \label{landalargo}
\end{figure}
Hence,   the    branch of minimal solutions is  given by  (the values at the boundary are $0$):
$$u_\lambda(2)=u_\lambda(3)=-W_0\left( - 2 \lambda \right) \quad\forall \, 0<\lambda\le \lambda^*.$$

This branch continues with a second branch of solutions
$$u^{\lambda}(2)=u^{\lambda}(3)=-W_{-1}\left( - 2 \lambda \right) \quad\forall \,0<\lambda< \lambda^*.$$
And, for $0<\lambda<\hat\lambda$, two more branches of solutions appear
$u^{3,\lambda}$ and $u^{4,\lambda}$ as described above.
See the bifurcation diagram  of $\lambda\mapsto u$ in~Figure~\ref{fig4puntosbueno}.
\begin{figure}[h]
\includegraphics[scale=0.5]{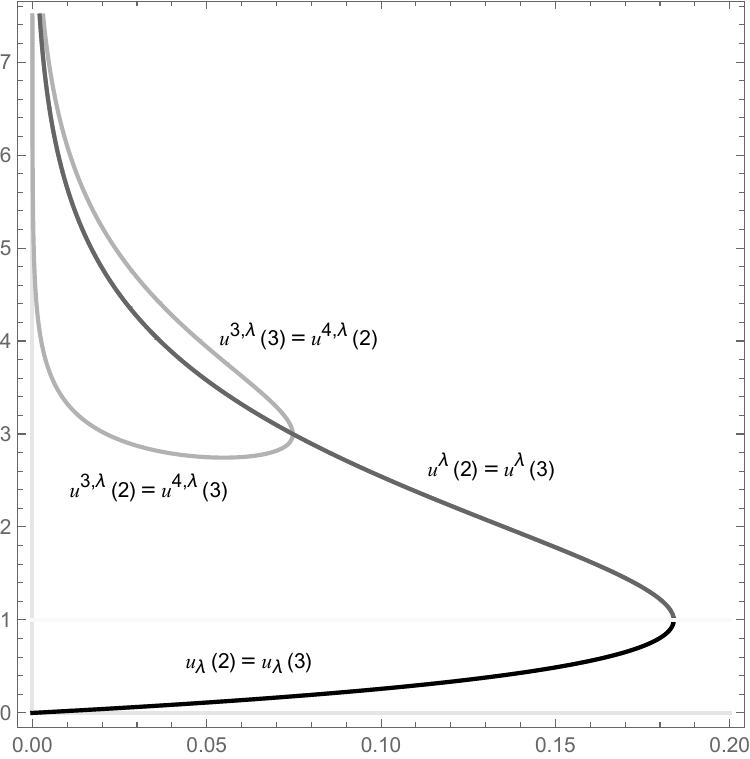}
\caption{Bifurcation diagram  for $\lambda\mapsto u(x)$, $x=2,3$, in Example~\ref{example4pt}. $\lambda^*=\frac{1}{2e}\approx 0.184$.}
\label{fig4puntosbueno}
\end{figure}

  If we change the weights to
$$w_{1,2}=w_{34} = b>0,\  w_{23}=a>0,$$ and $w_{i,j} = 0$ otherwise (see Figure~\ref{fig03}),
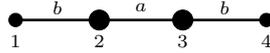
\begin{figure}[ht]
\centering
\begin{tikzpicture}[line cap=round,line join=round,>=triangle 45,x=1.1cm,y=1.1cm]
\draw [line width=1pt] (-2,0)-- (1,0);
\begin{scriptsize}
\draw [fill=black] (-2,0) circle (2.5pt);
\draw[color=black] (-2,-0.25) node {$1$};
\draw [fill=black] (-1,0) circle (3.75pt);
\draw[color=black] (-1,-0.25) node {$2$};
\draw[color=black] (-1.5,0.15) node {$b$};
\draw [fill=black] (0,0) circle (3.75pt);
\draw[color=black] (0,-0.25) node {$3$};
\draw[color=black] (-0.5,0.15) node {$a$};
\draw [fill=black] (1,0) circle (2.5pt);
\draw[color=black] (1,-0.25) node {$4$};
\draw[color=black] (0.5,0.15) node {$b$};
\end{scriptsize}
\end{tikzpicture}
\caption{Different weights (symmetric).}
\label{fig03}
\end{figure}
since there is symmetry in the space, there are also symmetric solutions, one is  the branch of minimal solutions, given by:
$$u^{1,\lambda}(2)=u^{1,\lambda}(3)=-W_0\left( - \frac{a+b}{b} \lambda \right) \quad\forall 0<\lambda\le \lambda^*.$$
Here, $\lambda_m(\Omega)=\frac{b}{a+b}$ and
$$\lambda^* = \lambda_m(\Omega)\frac{1}{e}=\frac{b}{a+b} \, \frac{1}{e}.$$
This branch continues with a second branch of  symmetric   solutions
$$u^{2,\lambda}(2)=u^{2,\lambda}(3)=-W_{-1}\left( - \frac{a+b}{b} \lambda \right) \quad\forall \ 0<\lambda< \lambda^*.$$ There are two more branches as before.
Observe that, letting $b\to +\infty$ in $\lambda^*$ we get  $\frac{b}{a+b}\frac{1}{e}\to \frac{1}{e}=\frac{s_0}{f(s_0)}$, which becomes optimal (see~\eqref{optimallambdamas}).
\hfill $\blacksquare$}
\end{example}

\begin{remark} \rm For the trivial weighted graph $\{1,2\}$  with weight $w_{12}=1$ and $\Omega=\{1\}$ (observe that $\Omega$ is not m-connected)
we have that the solutions to $(P(\lambda e^s))$ are (the value  at the boundary is $0$)
$$u_\lambda(2)=-W_0(-\lambda)$$
and
 $$u^\lambda(2) = -W_{-1}(-\lambda).$$
 \hfill $\blacksquare$
\end{remark}

We can extend the result of Example~\ref{example4pt} to the following  weighted graphs.

\begin{example}\label{kn}{\rm
Consider, for $n\ge 3$, the weighted complete graph $K_n=\{1,2,,...,n\}$, with weights for the edges on the $n$-cycle all equal to $a>0$,
$$w_{1,2}=w_{2,3}=...=w_{n-1,n}=w_{n,1}=a>0,$$
and the other weights equal to $c\ge 0$ (if $c=0$ we are dealing in fact with the $n$-cycle $C_n$). Take the weighted graph
$$\widehat K_n:=K_n\cup\{-1,-2,...,-n\}$$
with the above given weights between vertices in $K_n$ and with
$$w_{i,-i}=b>0,\ i=1,2,...,n,$$
and any other $w_{i,j}=0$, see Figure~\ref{k4c4} or Figure~\ref{petersen02} (observe, that we can take all the points $\{-1,-2,...,-n\}$ being the same, as in Figure~\ref{equivgraph01}).
\begin{figure}[ht]
\centering
\begin{tikzpicture}[line cap=round,line join=round,>=triangle 45,x=1.1cm,y=1.1cm]
\draw [line width=1pt] (-1,0)-- (0,1);
\draw [line width=1pt] (0,1)-- (1,0);
\draw [line width=1pt] (1,0)-- (0,-1);
\draw [line width=1pt] (0,-1)-- (-1,0);
\draw [line width=1pt] (-1,0)-- (-2,0);
\draw [line width=1pt] (1,0)-- (2,0);
\draw [line width=1pt] (0,1)-- (0,2);
\draw [line width=1pt] (0,-1)-- (0,-2);
\draw [line width=1pt] (-1,0)-- (1,0);
\draw [line width=1pt] (0,-1)-- (0,1);
\begin{scriptsize}
\draw [fill=black] (-2,0) circle (2.5pt);
\draw [fill=black] (-1,0) circle (3.75pt);
\draw [fill=black] (0,1) circle (3.75pt);
\draw [fill=black] (0,2) circle (2.5pt);
\draw [fill=black] (1,0) circle (3.75pt);
\draw [fill=black] (2,0) circle (2.5pt);
\draw [fill=black] (0,-1) circle (3.75pt);
\draw [fill=black] (0,-2) circle (2.5pt);

\draw[color=black] (-2,-0.25) node {$-1$};
\draw[color=black] (-1,-0.25) node {$1$};
\draw[color=black] (-1.5,0.15) node {$b$};

\draw[color=black] (1,-0.25) node {$3$};
\draw[color=black] (2,-0.25) node {$-3$};
\draw[color=black] (1.5,0.15) node {$b$};

\draw[color=black] (-0.35,2) node {$-4$};
\draw[color=black] (-0.25,1) node {$4$};
\draw[color=black] (0.15,1.5) node {$b$};

\draw[color=black] (-0.25,-1) node {$2$};
\draw[color=black] (-0.35,-2) node {$-2$};
\draw[color=black] (0.15,-1.5) node {$b$};

\draw[color=black] (0.6,0.6) node {$a$};
\draw[color=black] (0.6,-0.6) node {$a$};
\draw[color=black] (-0.6,0.6) node {$a$};
\draw[color=black] (-0.6,-0.6) node {$a$};

\draw[color=black] (-0.45,-0.15) node {$c$};
\draw[color=black] (0.15,0.45) node {$c$};

\end{scriptsize}
\end{tikzpicture}
\caption{$\widehat K_4$.}\label{k4c4}
\end{figure}
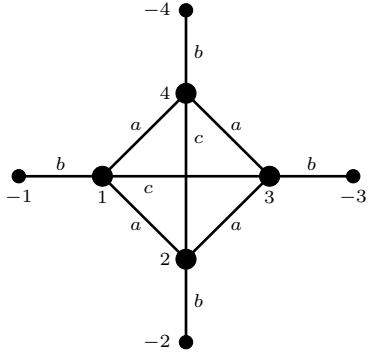

\begin{figure}[ht]
\begin{tikzpicture}
\begin{scriptsize}
    % Coordenadas para los v\'{e}rtices
    \node[fill=black, circle, inner sep=3.25pt] (1) at ({360/5 * (1 - 1)}:2) {};
    \node[fill=black, circle, inner sep=2.05pt] (2) at ({360/5 * (2 - 1)}:2) {};
    \node[fill=black, circle, inner sep=3.25pt] (3) at ({360/5 * (3 - 1)}:2) {};
    \node[fill=black, circle, inner sep=3.25pt] (4) at ({360/5 * (4 - 1)}:2) {};
    \node[fill=black, circle, inner sep=3.25pt] (5) at ({360/5 * (5 - 1)}:2) {};

    % Numeraci\'{o}n de los v\'{e}rtices
    \node[above right] at ({360/5 * (1 - 1)}:2.1) {4};
    \node[above right] at ({360/5 * (2 - 1)}:2) {5};
    \node[above right] at ({360/5 * (3 - 1)}:2.3) {1};
    \node[above right] at ({360/5 * (4 - 1)}:2.7) {2};
    \node[above right] at ({360/5 * (5 - 1)}:2.3) {3};

    % Etiquetas para las aristas con l\'{\i}neas m\'{a}s gruesas
    \path[draw, thick] (1) -- node[right=0.1] {$b$} (2);
    \path[draw, thick] (1) -- node[above] {$a$} (3);
    \path[draw, thick] (1) -- node[above] {$c$} (4);
    \path[draw, thick] (1) -- node[right] {$a$} (5);
    \path[draw, thick] (2) -- node[above] {$b$} (3);
    \path[draw, thick] (2) -- node[below] {$b$} (4);
    \path[draw, thick] (2) -- node[right] {$b$} (5);
    \path[draw, thick] (3) -- node[left] {$a$} (4);
    \path[draw, thick] (3) -- node[right] {$c$} (5);
    \path[draw, thick] (4) -- node[below] {$a$} (5);
\end{scriptsize}
\end{tikzpicture}
\caption{A graph equivalent to $\widehat K_4$ for the Gelfand problem. The points $\{-1,-2,-3,-4\}$ are joined at point $5$.}
\label{equivgraph01}
\end{figure}
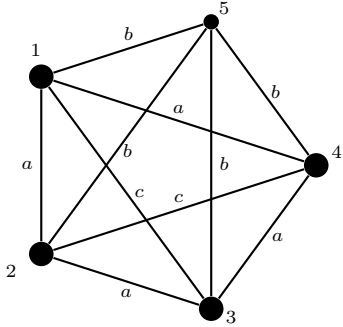

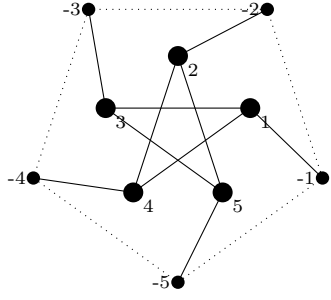
\begin{figure}[ht]
\centering
\begin{tikzpicture}
\begin{scriptsize}
  % Define the positions of the vertices (rotated 180 degrees)
  \foreach \i in {1,...,5}
  {
    \path (270+72*\i:2) coordinate (A\i);
    \path (270+72*\i+36:1) coordinate (B\i);
  }

  % Draw the edges of the outer pentagon
  \foreach \i in {1,...,5}
  {
    \pgfmathtruncatemacro{\nexti}{mod(\i,5)+1}
    \draw[dotted] (A\i) -- (A\nexti);
  }

  % Draw the edges of the inner star
  \foreach \i in {1,...,5}
  {
    \pgfmathtruncatemacro{\nexti}{mod(\i+2,5)+1}
    \draw (B\i) -- (B\nexti);
  }

  % Draw the edges connecting the outer pentagon to the inner star
  \foreach \i in {1,...,5}
  {
    \draw (A\i) -- (B\i);
  }

  % Draw the vertices and labels
  \foreach \i in {1,...,5}
  {
    \fill (A\i) circle (2.5pt);
    \fill (B\i) circle (3.75pt);
    \node[anchor=north] at (B\i ) {\hspace{0.4cm}\i};
    \node[anchor=east] at (A\i) {\hspace{0.5cm}-\i};
  }
\end{scriptsize}
\end{tikzpicture}
\caption{$\widehat K_5$ with $c=0$ in Example~\ref{kn}. The subgraph $\Omega=\{1,2,3,4,5\}$ is in fact $C_5$. It is the Petersen Graph if Dirichlet vertices (those of the boundary of $\Omega$) are joined as shown.}\label{petersen02}.
\end{figure}

\noindent Then, for $\Omega=K_n$, the Gefand problem $(P(\lambda e^s))$ has two  branches of symmetric solutions given by $u(x)=0$ for $x=-1,-2,-...,-n$, and
$$u(x)=-W\left(-\frac{2a+(n-3)c+b}{b}\lambda\right),\ x=1,2,...,n,$$
for  $0<\lambda\le \frac{b}{2a+(n-3)c+b}\frac{1}{e}$.
One of this is the minimal branch (see at the end of the example), $$u_\lambda(x)=-W_0\left(-\frac{2a+(n-3)c+b}{b}\lambda\right),\ x=1,2,...,n,$$
for $0\le\lambda\le\lambda^*=\frac{b}{2a+(n-3)c+b}\frac{1}{e}.$  And a second branch is
$$u^{2,\lambda}(x)=-W_{-1}\left(-\frac{2a+(n-3)c+b}{b}\lambda\right),\ x=1,2,...,n,$$
for $0<\lambda\le\lambda^*.$

Let us proof that the minimal branch of solutions is given by $$u_\lambda(x)=-W_0\left(-\frac{2a+(n-3)c+b}{b}\lambda\right),\quad x=1,2,...,n$$ We proof this for $\widehat K_4$ (see Figure~\ref{k4c4}), since the arguments of the general case are the same, and this is simpler and illustrative. A solution $x=u(1)$, $y=u(2)$, $z=u(3)$, $w=u(4)$ (the boundary values ar $0$) of $(P(\lambda e^s))$ satisfy
$$\left\{\begin{array}{l}
 x-\frac{a}{2a+c+b}y-\frac{c}{2a+c+b}z-\frac{a}{2a+c+b}w= \lambda e^x ,\\
\\
y-\frac{a}{2a+c+b}x -\frac{a}{2a+c+b}z-\frac{c}{2a+c+b}w= \lambda e^y ,\\
\\
z-\frac{c}{2a+c+b}x-\frac{a}{2a+c+b}y- \frac{a}{2a+c+b}w=  \lambda e^z ,\\
\\
w-\frac{a}{2a+c+b}x-\frac{c}{2a+c+b}y-\frac{a}{2a+c+b}z =  \lambda e^w .
\end{array}\right.
 $$
Adding the above for equations we deduce that
 $$\frac{b}{2a+c+b}(x+y+z+w)=\lambda(e^x+e^y+e^z+e^w).$$ Now, by convexity, $e^x+e^y+e^z+e^w \ge 4 e^{\frac{x+y+z+w}{4}}$, and, hence
 $$\frac{b}{2a+c+b}\frac{x+y+z+w}{4}\ge\lambda e^{\frac{x+y+z+w}{4}}.$$
 Then,
 $$-\frac{x+y+z+w}{4}e^{-\frac{x+y+z+w}{4}}\le-\frac{2a+c+b}{b}\lambda .$$
Therefore,
$$-W_0\left(-\frac{2a+c+b}{b}\lambda\right)\le \frac{x+y+z+w}{4}\le -W_{-1}\left(-\frac{2a+c+b}{b}\lambda\right).$$
Hence, $x=y=z=w=-W_0\left(-\frac{2a+c+b}{b}\lambda\right)$ gives the minimal solution.
\hfill $\blacksquare$}
\end{example}

\begin{remark}\label{remuni01}\rm
In the  previous example, for a fixed graph, there exists $k>1$, such that
\begin{equation}\label{remgi02}\frac{1}{d_x}\sum_{\tiny\begin{array}{c}y\in\partial_{m^G}\Omega\\ y\sim x\end{array}}w_{x,y}=\frac{1}{k}\quad\hbox{for each $x\in \Omega$,}
\end{equation}
(consequently, any $x\in\Omega$ is connected to a Dirichlet vertex),
and, for such situation, the same previous arguments give that the Gelfand problem $(P(\lambda e^s))$ (or changing $e^s$ by any other function $f\in \mathcal{C}^1([0,\infty[)$, strictly convex function and satisfying~\eqref{C1}, changing the Lambert $W$ function adequately attending point 3 in Remark~\ref{remgi}) has a branch of solutions given by
$$u(x)=-W\left(-k\lambda\right),\ x=1,2,...,|\Omega|,$$
for  $0<\lambda\le \frac{1}{ke}$,
being   $$u_\lambda(x)=-W_0\left(-k\lambda\right),\ x=1,2,...,|\Omega|,$$
for $0\le\lambda\le\lambda^*=\frac{1}{ke},$ the minimal solution.

Joint to the cases of Example~~\ref{kn}, this is also the case of $\Omega$ being any $n$-regular finite graph with edges weighted by $a>0$ and with each vertex joined to a Dirichlet vertex with and edge weighted by $b>0$ (see an example in Figure~\ref{Holt01}). In this case the weighted index of each $x\in\Omega$ is equal to $na+b$ and $$k=\frac{na+b}{b}.$$

\begin{figure}[ht]
    \centering
        \begin{tikzpicture}[scale=1.2]

% Define vertices in a circular arrangement
\foreach \i in {1,...,9}{
    \node[circle, fill=black, inner sep=2.5pt] (A\i) at (360/9*\i:2) {};  % Outer cycle
    \node[circle, fill=black, inner sep=2.5pt] (B\i) at (360/9*\i:1.2) {};  % Middle cycle
    \node[circle, fill=black, inner sep=2.5pt] (C\i) at (360/9*\i:0.6) {}; % Inner cycle
    \node[circle, fill=black, inner sep=1.75pt] (D\i) at (360/9*\i+15:2.8) {}; %ciclo Dirichlet
}

% Edges for the outer cycle
\foreach \i [count=\j from 2] in {1,...,8}{
    \draw[very thick] (A\i) -- (A\j);
}
\draw[very thick] (A9) -- (A1);

% Edges for the middle cycle
\foreach \i [count=\j from 2] in {1,...,8}{
    \draw[very thick] (B\i) -- (B\j);
}
\draw[very thick] (B9) -- (B1);

% Edges for the inner cycle
\foreach \i [count=\j from 2] in {1,...,8}{
    \draw[very thick] (C\i) -- (C\j);
}
\draw[very thick] (C9) -- (C1);

% Connections between different layers
\foreach \i in {1,...,9}{
    \draw[very thick] (A\i) -- (B\i);
    \draw[very thick] (B\i) -- (C\i);
}

% Cross connections between layers for 4-regularity
\draw[very thick] (A1) -- (C2);
\draw[very thick] (A2) -- (C3);
\draw[very thick] (A3) -- (C4);
\draw[very thick] (A4) -- (C5);
\draw[very thick] (A5) -- (C6);
\draw[very thick] (A6) -- (C7);
\draw[very thick] (A7) -- (C8);
\draw[very thick] (A8) -- (C9);
\draw[very thick] (A9) -- (C1);

% Cross connections con los puntos Diriclet

\draw[thin] (A2) -- (D1);
\draw[thin] (A3) -- (D2);
\draw[thin] (A4) -- (D3);
\draw[thin] (A5) -- (D4);
\draw[thin] (A6) -- (D5);
\draw[thin] (A7) -- (D6);
\draw[thin] (A8) -- (D7);
\draw[thin] (A9) -- (D8);
\draw[thin] (A1) -- (D9);

\draw[thin] (B2) -- (D1);
\draw[thin] (B3) -- (D2);
\draw[thin] (B4) -- (D3);
\draw[thin] (B5) -- (D4);
\draw[thin] (B6) -- (D5);
\draw[thin] (B7) -- (D6);
\draw[thin] (B8) -- (D7);
\draw[thin] (B9) -- (D8);
\draw[thin] (B1) -- (D9);

\draw[thin] (C2) -- (D1);
\draw[thin] (C3) -- (D2);
\draw[thin] (C4) -- (D3);
\draw[thin] (C5) -- (D4);
\draw[thin] (C6) -- (D5);
\draw[thin] (C7) -- (D6);
\draw[thin] (C8) -- (D7);
\draw[thin] (C9) -- (D8);
\draw[thin] (C1) -- (D9);

\end{tikzpicture}
\caption{$4$-regular  Holt Graph: larger dots and thicker edges.}\label{Holt01}.
\end{figure}
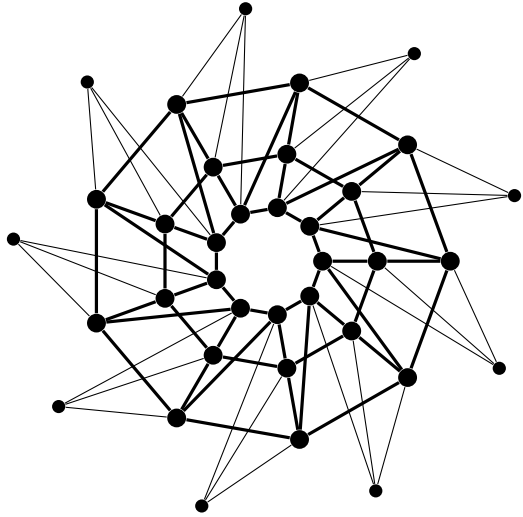

But there are many  other {\it non-so-regular} graphs satisfying such condition: For $\Omega=\{x_1,x_2,...,x_n\}$ being any connected weighted graph and $k>1$, set $\partial_{m^G}\Omega=\{x_0\}$ (Dirichlet vertices can be joined in one), and connect each $x_i$ in $\Omega$ to $x_0$ with and edge weighted by $$w_{x_i,x_0}=\frac{1}{k-1}\displaystyle \sum_{x_j\sim x_i}w_{x_i,x_j}.$$
Then the new graph satisfies~\eqref{remgi02} (weighed indices on vertices most be redefined).
\hfill $\blacksquare$
\end{remark}

\begin{example}\label{example3pointsa1} \rm   Let $G=(V,E)$ be the  graph $V= \{1,2,3\}$, with weights $$w_{12}=w_{23}=1,$$ and $w_{i,j} = 0$ otherwise.
\begin{figure}[ht]
\centering
\begin{tikzpicture}[line cap=round,line join=round,>=triangle 45,x=1.1cm,y=1.1cm]
\draw [line width=1pt] (-2,0)-- (0,0);

\begin{scriptsize}
\draw [fill=black] (-2,0) circle (2.5pt);
\draw[color=black] (-2,-0.25) node {$1$};
\draw [fill=black] (-1,0) circle (3.75pt);
\draw[color=black] (-1,-0.25) node {$2$};
\draw[color=black] (-1.5,0.15) node {$1$};
\draw [fill=black] (0,0) circle (3.75pt);
\draw[color=black] (0,-0.25) node {$3$};
\draw[color=black] (-0.5,0.15) node {$1$};
\end{scriptsize}
\end{tikzpicture}
\caption{Linear weighted graph in Example~\ref{example3pointsa1}.}\label{fig02a1}
\end{figure}
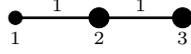 This is the weighted linear graph given in Figure~\ref{fig02a1}. Set $\Omega:= \{2,3\}$. We have that $u$ is a solution of the Gelfand problem $(P(\lambda e^s))$ if
$u(1)= 0$, and
\begin{equation}\label{problem3Exam5pta1cc}\left\{\begin{array}{ll}
u(3)=2\left(u(2) - \lambda  \, e^{u(2)}\right),
\\ \\
 u(2)  = u(3)-\lambda  \, e^{u(3)}.
\end{array}\right.
\end{equation}
Solutions of this problem are illustrated in Figure~\ref{dostres001}.
\begin{figure}[ht]
  \centering
  \begin{subfigure}{0.2\textwidth}
    \centering
    \includegraphics[width=\linewidth]{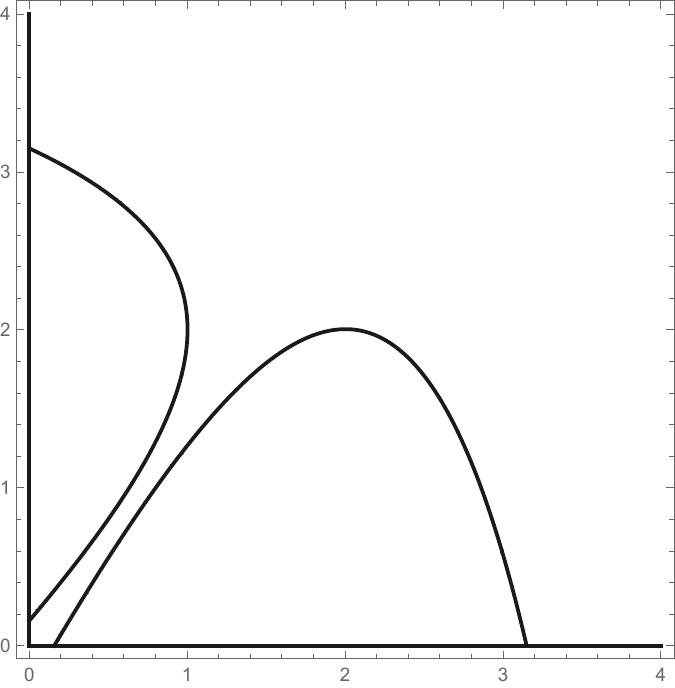}
    \caption{\tiny $\lambda>\lambda^*$.}
    %\label{fig:subfig123005}
  \end{subfigure}
            \hspace{0.5cm}
     \begin{subfigure}{0.2\textwidth}
    \centering
    \includegraphics[width=\linewidth]{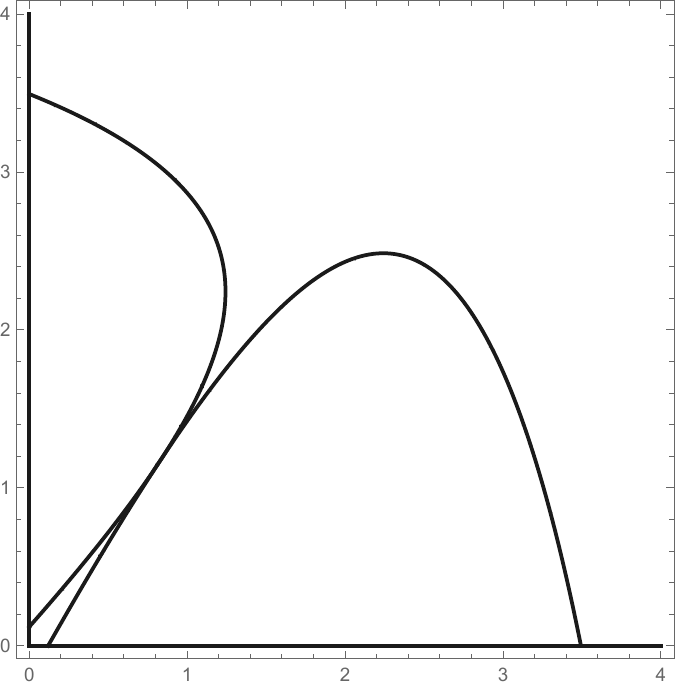}
    \caption{\tiny  $\lambda=\lambda^*$.}
    %\label{fig:subfig223006}
            \end{subfigure}
            \hspace{0.5cm}
   \begin{subfigure}{0.2\textwidth}
    \centering
    \includegraphics[width=\linewidth]{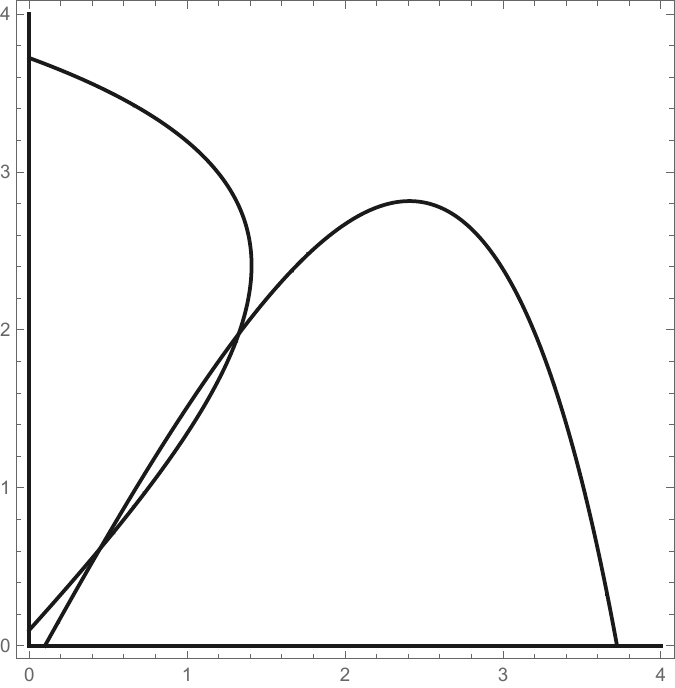}
    \caption{\tiny  $0<\lambda=\lambda_1<\lambda^*$.}
    %\label{fig:subfig223007}
  \end{subfigure}

 \vspace{0.5cm}

  \begin{subfigure}{0.2\textwidth}
    \centering
    \includegraphics[width=\linewidth]{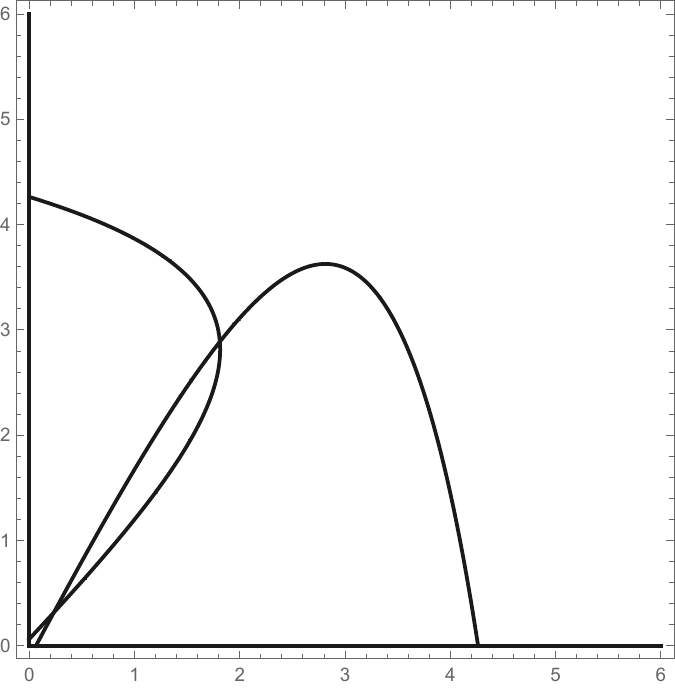}
    \caption{\tiny  $0<\lambda=\lambda_2<\lambda_1$.}
    %\label{fig:subfig1234008}
  \end{subfigure}
            \hspace{0.5cm}
     \begin{subfigure}{0.2\textwidth}
    \centering
    \includegraphics[width=\linewidth]{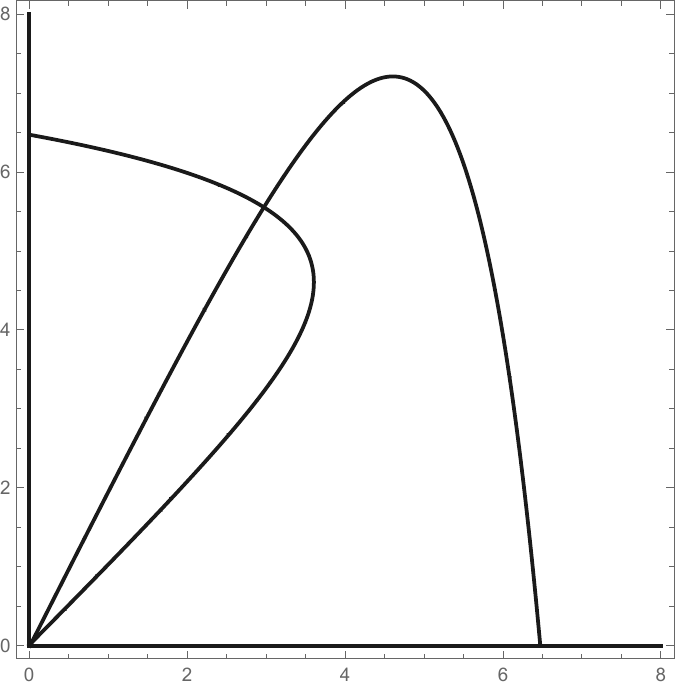}
    \caption{\tiny $0<\lambda=\lambda_3<\lambda_2$.}
    %\label{fig:subfig2234009}
  \end{subfigure}
  \hspace{0.5cm}
   \begin{subfigure}{0.2\textwidth}
    \centering
    \includegraphics[width=\linewidth]{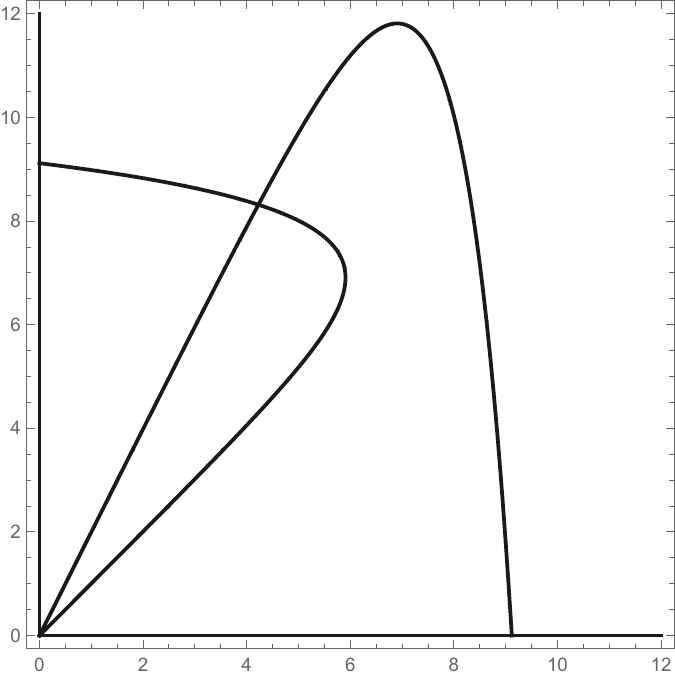}
    \caption{\tiny $0<\lambda=\lambda_4<\lambda_3$.}
    %\label{fig:subfig22340010}
  \end{subfigure}

  \caption{Solutions of~\eqref{problem3Exam5pta1cc}.}
  \label{dostres001}
\end{figure}
  There exist only  two branches of solutions for $0<\lambda<\lambda^*\simeq 0.106159$, $u_\lambda<u^\lambda$. The bifurcation diagram is given in Figure~\ref{fig5puntosdoble}.  One can see that $-W_0(-3\lambda)\le u_\lambda\le -W_0(-4\lambda)$ if $0<\lambda<\frac{1}{4e}$, hence $0.091970\simeq\frac{1}{4e}<\lambda^*$.

  Let us see, by using the implicit function theorem, that $$\lambda\mapsto (u_\lambda(2),u_\lambda(3)) \in C^\infty(]0,\lambda^*[:\mathbb{R}^2).$$  Indeed, for simplicity let us call $x=u(2)$ and  $y=u(3)$ in~\eqref{problem3Exam5pta1cc}, then such system can be written as
\begin{equation}\label{cinf001}\left\{\begin{array}{ll}
 F_1(x,y,\lambda):=2\left(x - \lambda  \, e^{x}\right)-y=0,
\\ \\
F_2(x,y,\lambda):=y-\lambda  \, e^{y}-x=0.
\end{array}\right.
\end{equation}
Where $F_i$, $i=1,2$, are $C^\infty$ functions in $\mathbb{R}^3$. Now since, for $\lambda=\lambda^*$ there is only one solution $x^*:=x(\lambda^*), y^*:=y(\lambda^*)$ to such system. Then, for
$$
H(x,y, \lambda):=\hbox{det}\left(\begin{array}{ll}
\frac{\partial F_1}{\partial x}(x,y,\lambda)&\frac{\partial F_1}{\partial y}(x,y,\lambda)\\
\\
\frac{\partial F_2}{\partial x}(x,y,\lambda)&\frac{\partial F_2}{\partial y}(x,y,\lambda)
\end{array}
\right)=(2-2\lambda e^{x})(1-\lambda e^{y})-1,
$$
Since for $\lambda > \lambda^*$ the system \eqref{cinf001} has not solution, by the implicit function theorem, we have $$H(x^*,y^*,\lambda^*)=0,$$
that is,
$$(2-2\lambda^*e^{x^*})(1-\lambda^*e^{y^*})-1=0.$$
Now, since for $0<\lambda<\lambda^*$ we have that there is a solution $x(\lambda):=u_\lambda(2)<x^*$, $y(\lambda):=u_\lambda(3)<y^*$ of system~\eqref{cinf001}, we have that
$$2-2\lambda e^{x(\lambda)}>2-2\lambda^* e^{x^*}$$
and
$$1-\lambda e^{y(\lambda)}>1-\lambda^* e^{y^*},$$
and we have that $1-\lambda^* e^{x^*}>0$ and $1-\lambda^* e^{y^*}>0$;  therefore
$$H(x(\lambda),y(\lambda), \lambda)=(2-2\lambda e^{x(\lambda)})(1-\lambda e^{y(\lambda)})-1>(2-2\lambda^* e^{x^*})(1-\lambda^* e^{y^*})-1=0.$$ Hence, by the implicit function theorem, we get that $\lambda\mapsto (u_\lambda(2),u_\lambda(3))$ is a $C^\infty$ function in $]0,\lambda^*[$.
\hfill $\blacksquare$
\end{example}

 In the next example we consider  a linear graph like in Example~\ref{example4pt} but without symmetry.

\begin{example}\label{ejemnosim}{\rm Let $G=(V,E)$ be the  graph $V= \{1,2,3,4\}$, with weights $$w_{12}=w_{23}=1,\ w_{34}=2$$  and $w_{i,j} = 0$ otherwise (see Figure~\ref{fig03esla8}).
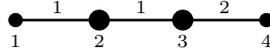
\begin{figure}[ht]
\centering
\begin{tikzpicture}[line cap=round,line join=round,>=triangle 45,x=1.1cm,y=1.1cm]
\draw [line width=1pt] (-2,0)-- (1,0);
\begin{scriptsize}
\draw [fill=black] (-2,0) circle (2.5pt);
\draw[color=black] (-2,-0.25) node {$1$};
\draw [fill=black] (-1,0) circle (3.75pt);
\draw[color=black] (-1,-0.25) node {$2$};
\draw[color=black] (-1.5,0.15) node {$1$};
\draw [fill=black] (0,0) circle (3.75pt);
\draw[color=black] (0,-0.25) node {$3$};
\draw[color=black] (-0.5,0.15) node {$1$};
\draw [fill=black] (1,0) circle (2.5pt);
\draw[color=black] (1,-0.25) node {$4$};
\draw[color=black] (0.5,0.15) node {$2$};
\end{scriptsize}
\end{tikzpicture}
\caption{Non symmetric graph in Example~\ref{ejemnosim}.}
\label{fig03esla8}
\end{figure}  In this case the branches $\lambda\mapsto u$ for the solutions of the Gelfand problem $(P(\lambda e^s))$ are given in Figure~\ref{PlostEx45_01}.
 \begin{figure}[h]
\includegraphics[scale=0.3]{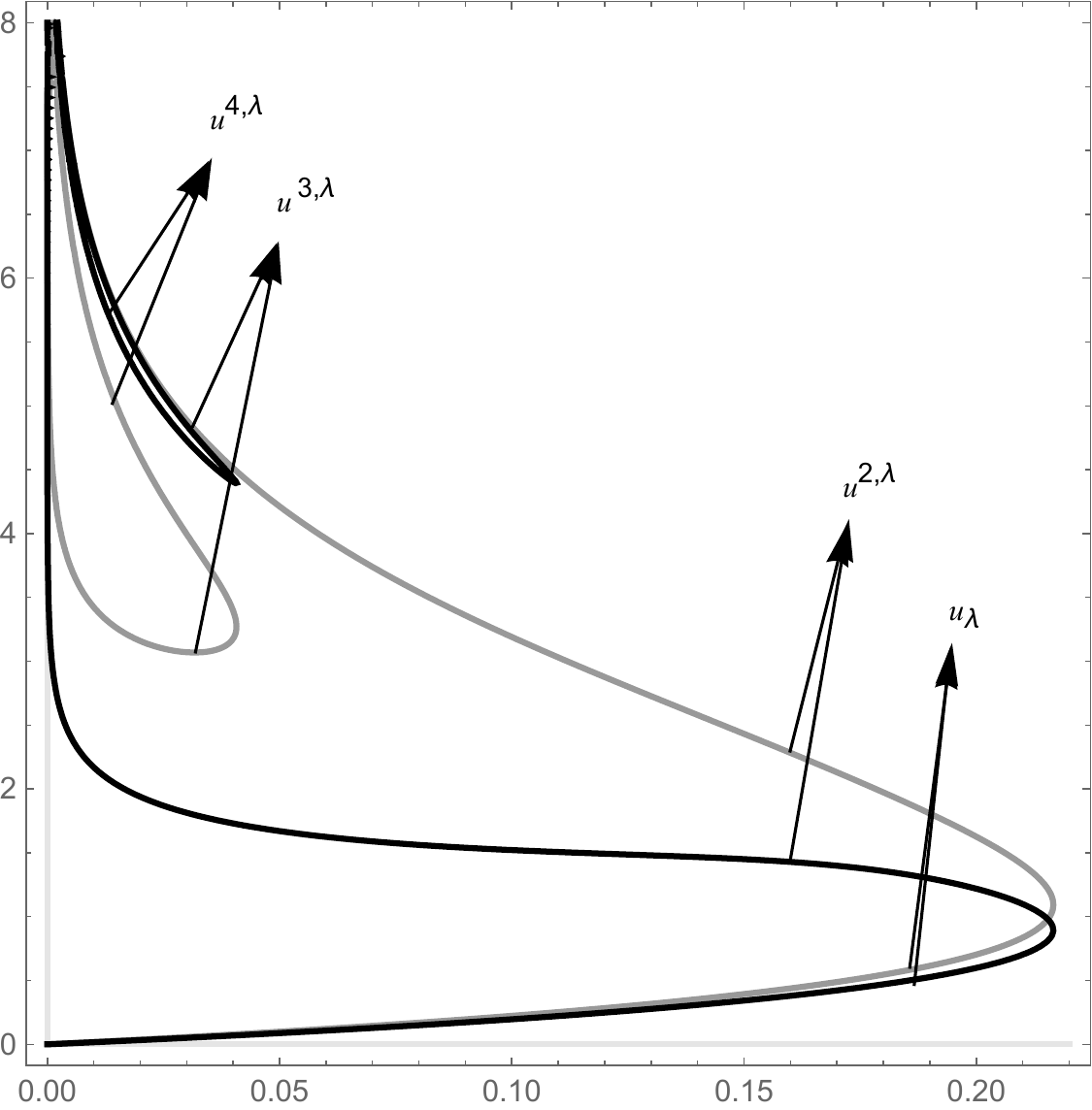}
\caption{Bifurcation diagram  for $\lambda\mapsto u$ in Example~\ref{ejemnosim} with a non symmetric linear weighed graph. $u(3)$ in black, $u(2)$ in gray. Four solutions: $u_\lambda$, $u^{2,\lambda}$, $u^{3,\lambda}$, $u^{4,\lambda}$.}
\label{PlostEx45_01}
\end{figure}
\hfill $\blacksquare$ }
\end{example}

\begin{example}\label{5puntosex}{\rm   Let $G=(V,E)$ the graph $V= \{1,2,3,4,5\}$, with weights $w_{i,i+1} = 1$, $i=1,2,3,4$ and $w_{i,j} = 0$ otherwise (see Figure~\ref{fig04}).
\begin{figure}[ht]
\centering
\begin{tikzpicture}[line cap=round,line join=round,>=triangle 45,x=1.1cm,y=1.1cm]
\draw [line width=1pt] (-2,0)-- (2,0);
\begin{scriptsize}
\draw [fill=black] (-2,0) circle (2.5pt);
\draw[color=black] (-2,-0.25) node {$1$};
\draw [fill=black] (-1,0) circle (3.75pt);
\draw[color=black] (-1,-0.25) node {$2$};
\draw[color=black] (-1.5,0.15) node {$1$};
\draw [fill=black] (0,0) circle (3.75pt);
\draw[color=black] (0,-0.25) node {$3$};
\draw[color=black] (-0.5,0.15) node {$1$};
\draw [fill=black] (1,0) circle (3.75pt);
\draw[color=black] (1,-0.25) node {$4$};
\draw[color=black] (0.5,0.15) node {$1$};
\draw [fill=black] (2,0) circle (2.5pt);
\draw[color=black] (2,-0.25) node {$5$};
\draw[color=black] (1.5,0.15) node {$1$};
\end{scriptsize}
\end{tikzpicture}
\caption{Linear weighted graph in Example~\ref{5puntosex}.}\label{fig04}
\end{figure}
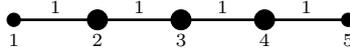 Let $\Omega:= \{2,3,4\}$.
Now, $u$ is a solution of  $(P(\lambda e^s))$ if and only
 $u(1) = u(5) = 0$,
\begin{equation}\label{problem3Exam5ptcc}\left\{\begin{array}{ll}
u(2) - \frac12 u(3) = \lambda  \, e^{u(2)},
\\ \\
u(3) - \frac12u(2)-\frac12u(4) = \lambda  \, e^{u(3)},
\\ \\
u(4) - \frac12 u(3) = \lambda  \, e^{u(4)}.
\end{array}\right.
\end{equation}
There exist two branches of symmetric solutions for $0<\lambda<\lambda^*\simeq 0.106159$, $u_\lambda<u^\lambda$.    The bifurcation diagram is given in Figure~\ref{fig5puntosdoble}.
 \begin{figure}[h]
\includegraphics[scale=0.4]{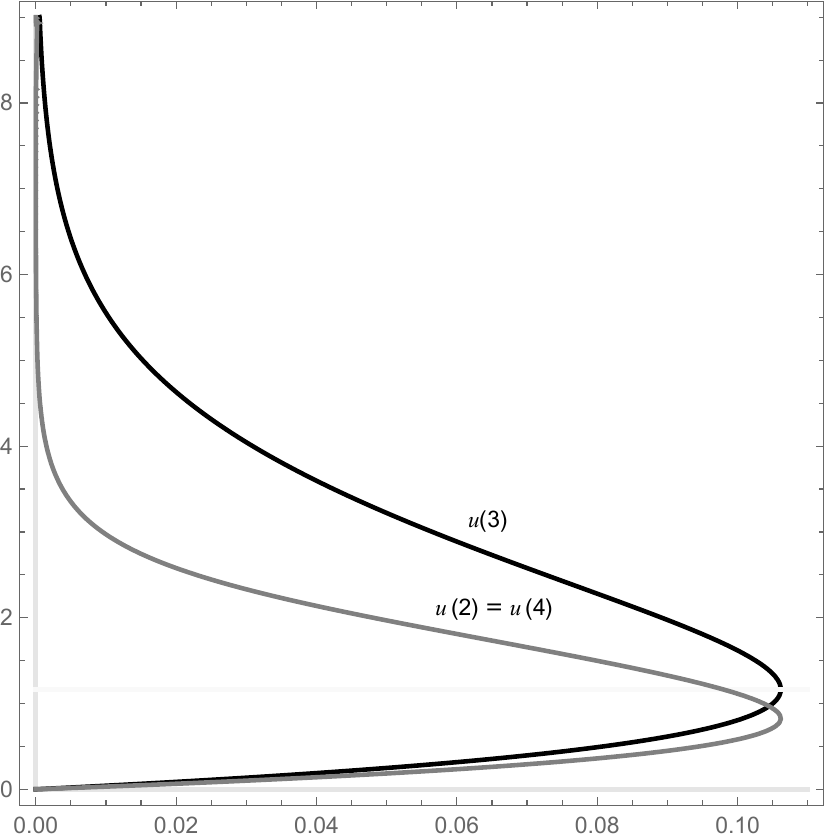}
\caption{Bifurcation diagram  for $\lambda\mapsto u$ in Example~\ref{example3pointsa1} (the branch of  solutions are given by $u(2)$ and $ u(3)$) and in Example~\ref{5puntosex} (the branch of  solutions are given by $u(2)=u(4)$ and $u(3)$). $\lambda^*\simeq 0.106159$, $u^*(3)\simeq1.164771$. }
\label{fig5puntosdoble}
\end{figure}
\hfill $\blacksquare$ }
\end{example}

 In the next example we use the approximation method  given in the proof of Theorem~\ref{existence} to characterize the minimal solutions.

\begin{example}\label{elsobre01}\rm Consider the graph given in Figure~\ref{elsobre02}, with $\Omega=\{1,2,3,4,5\}$, and all the edges weighted by $1$.
\begin{figure}[ht]
\centering
\begin{tikzpicture}[line cap=round,line join=round,>=triangle 45,x=1.1cm,y=1.1cm]
\draw [line width=1pt] (-1,0)-- (0,1);
\draw [line width=1pt] (0,1)-- (1,0);
\draw [line width=1pt] (1,0)-- (0,-1);
\draw [line width=1pt] (0,-1)-- (-1,0);
\draw [line width=1pt] (-1,0)-- (-2,0);
\draw [line width=1pt] (1,0)-- (2,0);
\draw [line width=1pt] (0,1)-- (0,2);
\draw [line width=1pt] (0,-1)-- (0,-2);
\draw [line width=1pt] (-1,0)-- (1,0);
\draw [line width=1pt] (0,-1)-- (0,1);
\begin{scriptsize}
\draw [fill=black] (-2,0) circle (2.5pt);
\draw [fill=black] (-1,0) circle (3.75pt);
\draw [fill=black] (0,1) circle (3.75pt);
\draw [fill=black] (0,2) circle (2.5pt);
\draw [fill=black] (1,0) circle (3.75pt);
\draw [fill=black] (2,0) circle (2.5pt);
\draw [fill=black] (0,-1) circle (3.75pt);
\draw [fill=black] (0,-2) circle (2.5pt);
\draw [fill=black] (0,0) circle (3.75pt);

\draw[color=black] (-2,-0.25) node {$-1$};
\draw[color=black] (-1,-0.25) node {$1$};

\draw[color=black] (1,-0.25) node {$3$};
\draw[color=black] (2,-0.25) node {$-3$};

\draw[color=black] (-0.35,2) node {$-4$};
\draw[color=black] (-0.25,1) node {$4$};

\draw[color=black] (-0.25,-1) node {$2$};
\draw[color=black] (-0.35,-2) node {$-2$};

\draw[color=black] (-0.15,-0.25) node {$5$};

\end{scriptsize}
\end{tikzpicture}
\caption{Graph in Example~\ref{elsobre01}, $d_x=4$ for all $x\in\Omega=\{1,2,3,4,5\}$. $\partial_{m^G}\Omega=\{-1,-2,-3,-4\}$.}\label{elsobre02}
\end{figure}
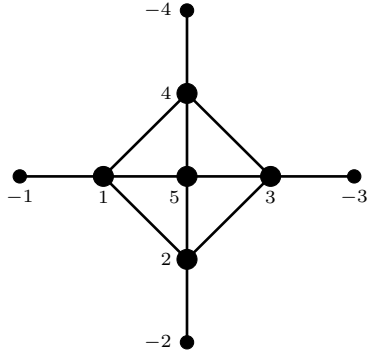
 The Gelfand problem $(P(\lambda e^s))$ has  a branch of solutions given by ($u=0$ on the boundary) $u(1)=u(2)=u(3)=u(4)$ and
\begin{equation}\label{elsobre03}\left\{\begin{array}{ll}
2 u(1)- u(5)= 4\lambda  \, e^{u(1)},
\\ \\
 u(5)  - u(1)=-\lambda  \, e^{u(5)}.
\end{array}\right.
\end{equation}
And the minimal branch of solutions satisfies the above properties. Indeed, the approximation method used in  the proof of Theorem~\ref{existence} to get the minimal solutions can be written in the following way:
$$\left(\begin{array}{c}u^0(1)\\ u^0(2)\\ u^0(3)\\ u^0(4)\\ u^0(5)\end{array}\right)=\left(\begin{array}{c}0\\ 0\\ 0\\ 0\\ 0\end{array}\right),$$ and, for $i=1,2,3,...$,
\begin{equation}\label{recseq01}\left(\begin{array}{c}u^i(1)\\ u^i(2)\\ u^i(3)\\ u^i(4)\\ u^i(5)\end{array}\right)=\lambda A^{-1}\left(\begin{array}{c}d_1e^{u^{i-1}(1)}\\ d_2e^{u^{i-1}(2)}\\ d_3e^{u^{i-1}(3)}\\ d_4e^{u^{i-1}(4)}\\ d_5 e^{u^{i-1}(5)}\end{array}\right),
\end{equation}
where
$$A=\left(\begin{array}{ccccc}
4&-1&0&-1&-1\\
-1&4&-1&0&-1\\
0&-1&4&-1&-1\\
-1&0&-1&4&-1\\
-1&-1&-1&-1&4
\end{array}\right).$$
 Now,
$$A^{-1}=\frac{1}{384}\left(\begin{array}{ccccc}
160&80&64&80&96\\
80&160&80&64&96\\
64&80&160&80&96\\
80&64&80&160&96\\
96&96&96&96&192
\end{array}\right),$$
and $d_1=d_2=d_3=d_4=d_5=4$; then   it is easy to see, from the recursive expression~\eqref{recseq01},  that
$$u^i(1)= u^i(2)= u^i(3)= u^i(4).$$ Now, since
$$u_\lambda(j)=\lim_{i\to \infty}u^{i}(j),\quad j=1,2,3,4,5,$$
we get that $u_\lambda$ satisfies $u_\lambda(1)=u_\lambda(2)=u_\lambda(3)=u_\lambda(4)$ and~\eqref{elsobre03}.

 The recursive sequence~\eqref{recseq01} is reminiscent of the power tower sequences for obtaining the Lambert $W$ function (see for example~\cite[Theorem 8]{javiertoledo}). This method can be easly implemented to approximate minimal solutions of  Gelfand type problems in any graph.
For example, $u^2$ is given by
$$\left(\begin{array}{c}u^2(1)\\ u^2(2)\\ u^2(3)\\ u^2(4)\\ u^2(5)\end{array}\right)=\left(\begin{array}{c}4\lambda e^{5\lambda}+\lambda e^{6\lambda}\\ 4\lambda e^{5\lambda}+\lambda e^{6\lambda}\\ 4\lambda e^{5\lambda}+\lambda e^{6\lambda}\\ 4\lambda e^{5\lambda}+\lambda e^{6\lambda}\\ 4\lambda e^{5\lambda}+2\lambda e^{6\lambda}\end{array}\right).$$
\hfill $\blacksquare$
\end{example}

  In all the above examples  we see that,    there is  a   continuum branch of the solutions $u^\lambda$ such that
\begin{equation}\label{grad02}  \lim_{\lambda\to 0^+}||u^\lambda||_\infty= +\infty.
\end{equation}

   Let us see that this is a general property.

   The proof of the following result is a consequence of Brouwer and Leray-Schauder degree theory; we use as a reference the book of Ambrosetti and Arcoya~\cite{AmbAr} for notation and results.

\begin{theorem}\label{grad01NEW}   Let $G = (V(G),E(G))$ be a weighted graph and $\Omega \subset V(G)$ finite and $m^G$-connected.
Let $f\in \mathcal{C}^1([0,\infty[)$ be   a strictly convex function satisfying~\eqref{H}.
Set $$S^*=\{(\lambda,u)\in[0,\lambda^*]\times  L^\infty(\Omega_m, \nu):u \hbox{ is a solution of }(P(\lambda f)) \}.$$
Then, for each $0<\epsilon<\lambda^*$, there exists $0<\lambda<\epsilon$ such that the problem
$(P(\lambda f))$
has a   solution $ u^\lambda$  with
$$||u^\lambda||_{ L^\infty(\Omega_m, \nu)}=g_2^{-1}(\epsilon),$$
where $g^{-1}_2$ is given in Remark~\ref{remgi} (and satisfies $\lim_{s\to 0^+}g_2^{-1}(s)= +\infty$).
Moreover, the   connected component of $S^*$ that contains   $(0,0)$ also contains $(\lambda,u^\lambda)$.
\end{theorem}
\begin{proof}  We can assume that $\Omega=\{1,2,...,n\}$ and that $\partial\Omega=\{0\}$.
    Observe first that, for $d_i=\sum_{j=0}^{n}w_{ij}$, and $x=(x_1,x_2,...,x_n)=(u(1),u(2),...,u(n))$, the problem
$(P(\lambda f))$ can be written as
\begin{equation}\label{grad04}\Phi(\lambda,x):=x-T(\lambda,x)=0,
\end{equation}
for   $T=(T_1,T_2,...,T_n)$ given by the smooth functions
$$T_i(\lambda,x)= \sum_{j=1}^{n}\frac{w_{ij}}{d_i}x_j+\lambda f(x_i),\quad i=1,2,...,n,
$$
being  $T$ is compact.  Note that by Corollary \ref{BoundI}, if  $$(u(0),u(1),u(2),...,u(n))=(0,x_1,x_2,...,x_n),$$ is a solution of problem
$(P(\lambda f))$ we have $$0\leq x_i \leq g^{-1}_2(\lambda).$$

For $0<\epsilon<\lambda^*$, set the open bounded set in $\mathbb{R}^n$ given by $$A=]-1,g^{-1}_2(\epsilon)[\times ]-1,g^{-1}_2(\epsilon)[\times ...\times ]-1,g^{-1}_2(\epsilon)[.$$
Since $0$ is the unique solution of~\eqref{grad04} for $\lambda=0$, we have that
the degree
$$\hbox{deg}(\Phi(0,.),A,0)
=\hbox{sign}\left(\hbox{det}(I-J_T(0,0)\right)\neq 0,$$
where  $J_T(0,0)$ is   the Jacobian  matrix of $T(0,x)$ at $x=0$.

Suppose that the connected component  $C$ of $$S:=\{(\lambda,x)\in[0,\lambda^*+1]\times \overline{A}:\Phi(\lambda,x)=0 \}$$ containing $(0,0)$ does not intersect $[0,\lambda^*+1]\times \partial A$. Then, arguing as in the proof of~\cite[Theorem 4.3.4]{AmbAr} we obtain  an open bounded set $B$ in $[0,\lambda^*+1]\times \mathbb{R}^n$ with
$C\subset B$ such that, on account of the general homotopy property (see~~\cite[Proposition 4.2.5]{AmbAr}),
$$0\neq\hbox{deg}(\Phi(0,.),B_\lambda,0)
=\hbox{deg}(\Phi(\lambda^*+1,.),B_{\lambda^*+1},0),$$
(where $B_\lambda=\{x:(\lambda,x)\in B\}$) which implies that~\eqref{grad04} has a solution  for $\lambda=\lambda^*+1$, which is false.

Therefore, $C$ intersects the set $[0,\lambda^*+1]\times\partial A$, but, since~\eqref{grad04} has no solution on  $[\epsilon,\lambda^*+1]\times\partial A$ and the only solution at $\lambda=0$ is $0$, we have that it has a solution on $]0,\epsilon[\times\partial A$,  from where the result follows, since   any solution must be non-negative.
\end{proof}

 In the next example we consider  the graph defined in Example \ref{example4pt} but for the Gelfand-type problem $(P(\lambda f))$ in which we consider a  non-convex function $f(s)$ (instead of the convex function $f(s)=e^s$) for which $\lambda\mapsto u_\lambda$ is not continuous.

\begin{example}\label{exnocont} \rm Let $G$ be the graph given in Example \ref{example4pt} and  the same $\Omega$ given there, and consider the Gelfand-type problem $(P(\lambda f))$ with  $f(s)=s^4-10s^3+24s^2+36s+1$, which satisfies~\eqref{H} and is not convex.     For this problem we have two symmetric solutions, $u_\lambda$ and $u^\lambda$, and there are up to six solutions for some interval of $\lambda$.   The  minimal branch $\lambda\to u_\lambda$ is not continuous. The   bifurcation diagram is given in Figure~\ref{no_continuo_++}.
\begin{figure}[ht]
\includegraphics[scale=0.5]{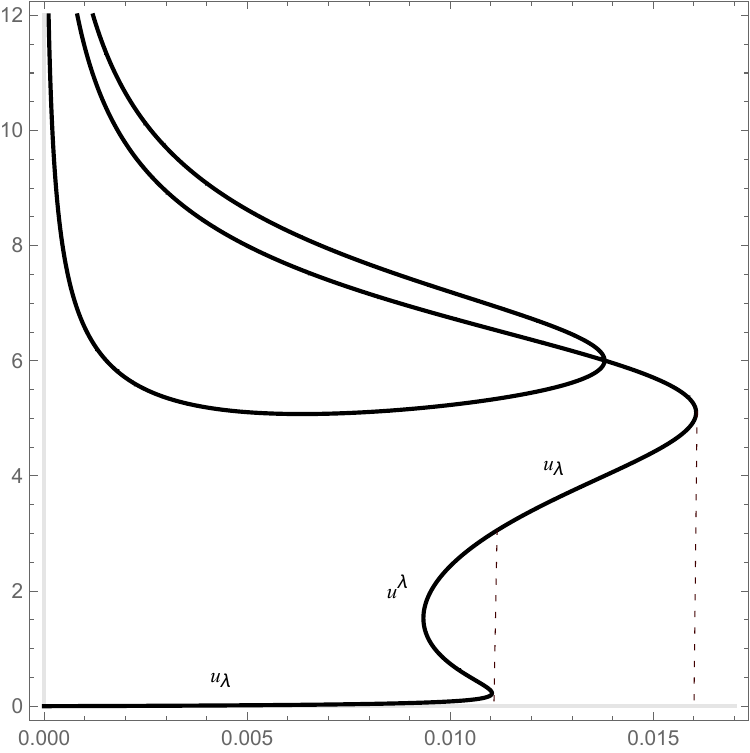}
\caption{Branches of $(\lambda, u)$ in Example~\ref{exnocont}. $(\lambda,  u_\lambda)$ is not continuous. $\lambda^*\simeq 0.0161546$, $u^*\simeq5.1007955$.}
\label{no_continuo_++}
\end{figure}
\hfill $\blacksquare$
\end{example}

 In the next example we see that there is a convex function $f$ for which there are infinitely many solutions for  $(P(\lambda^* f))$.

\begin{example}\label{infsols01}\rm  Let $G$ be the graph given in Example \ref{example4pt} and  the same $\Omega$ given there, and consider the Gelfand-type problem $(P(\lambda f))$ with $f$ the $\mathcal{C}^1$ and convex function$$f(s)=\left\{
\begin{array}{ll}
(s-1)^2+2(s-1)+2,& 0\le s<1,\\[6pt]
2s, & 1\le s\le 2,\\[6pt]
(s-2)^2+2(s-2)+2, & s>2.
\end{array}
\right.
$$
The solutions must satisfy
\begin{equation}\label{pb000012infsol}\left\{\begin{array}{ll}
u(2)  -\frac{1}{2}u(3) = \lambda  \, f(u(2)),
\\ \\
u(3)  -\frac{1}{2}u(2) = \lambda  \, f(u(3)).
\end{array}\right.
\end{equation}
This problem has two symmetric branches of solutions, $u_\lambda$ and $u^\lambda$. Problem  $(P(\lambda^* f))$ has infinite solutions, see Figure~\ref{landalargo01}.
\begin{figure}[ht]
  \centering
  \begin{subfigure}{0.25\textwidth}
    \centering
    \includegraphics[width=\linewidth]{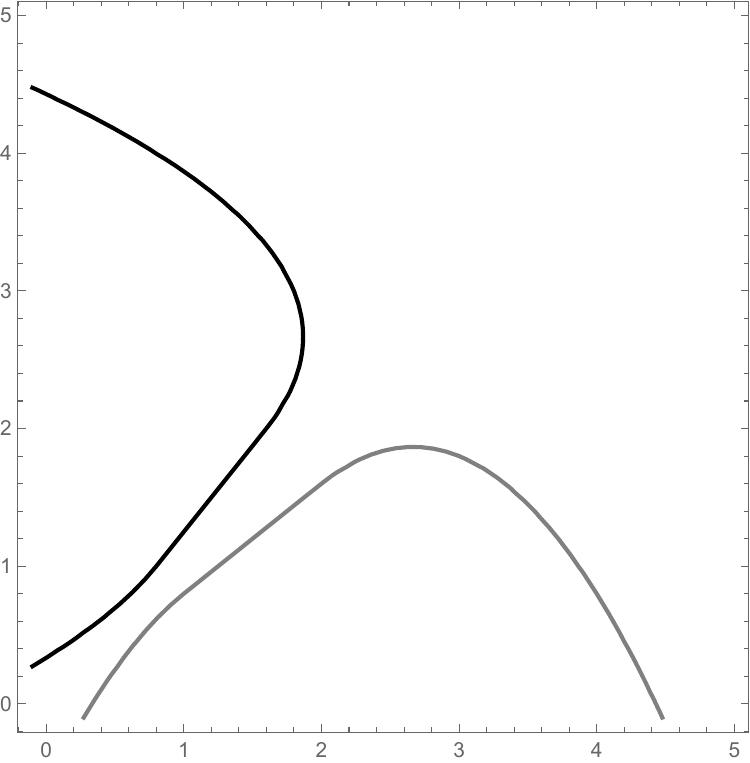}
    \caption{No solutions for $\lambda>\lambda^*$.}
    %\label{fig:subfig1101}
  \end{subfigure}
  \hspace{0.5cm}
   \begin{subfigure}{0.25\textwidth}
    \centering
    \includegraphics[width=\linewidth]{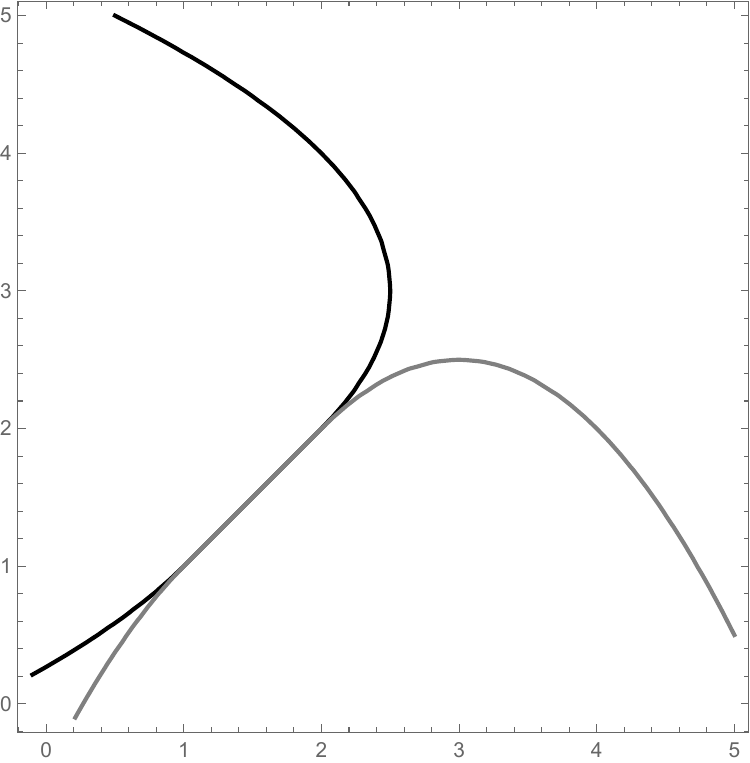}
    \caption{Infinitely many solutions for $\lambda^*=0.25$.}
    %\label{fig:subfig2102}
  \end{subfigure}
  \hspace{0.5cm}
  \begin{subfigure}{0.25\textwidth}
    \centering
    \includegraphics[width=\linewidth]{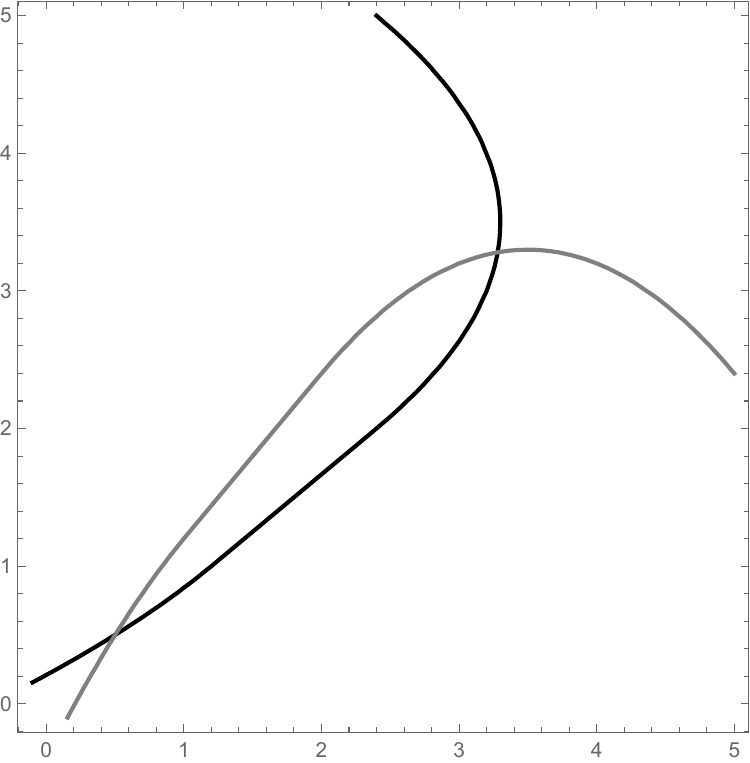}
    \caption{Two solutions for $0< \lambda<\lambda^*$.}
    %\label{fig:subfig3104}
  \end{subfigure}
 \caption{Solutions  of~\eqref{pb000012infsol}.}
  \label{landalargo01}
\end{figure}
The   bifurcation diagram is given in Figure~\ref{infinitasols}.
\begin{figure}[h]
\includegraphics[scale=0.5]{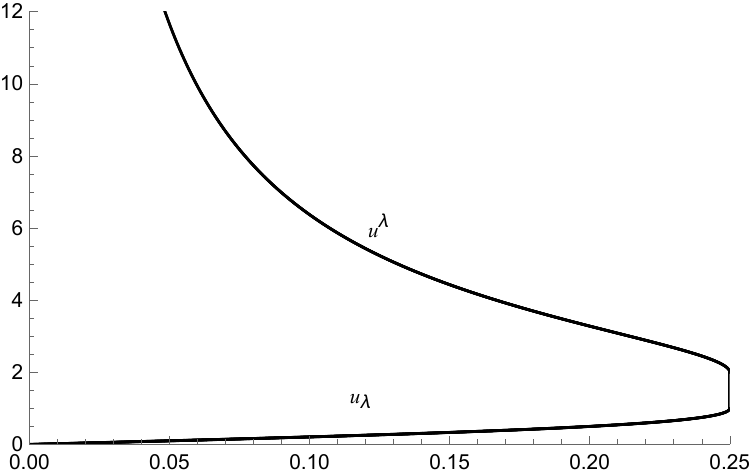}
\caption{Branches of $(\lambda, u)$ in Example~\ref{infsols01}. $\lambda^*\simeq 0.25$.}
\label{infinitasols}
\end{figure}
\hfill $\blacksquare$
\end{example}

\subsection{The  problem for power-like functions}

  Let us now see an example   for $f(s) = (1 +s)^2$.  In this case we have that
$$
g^{-1}_1(\lambda)=\frac{2\lambda}{1-2\lambda+\sqrt{1-4\lambda}}\leq u \leq  g^{-1}_2(\lambda)=\frac{1-2\lambda+\sqrt{1-4\lambda}}{2\lambda},
$$
for any   solution to $u$  to $(P(\lambda(1+s)^2))$, $0 < \lambda \le  \lambda^*\le\frac{\lambda_m(\Omega)}{4}$. See Figure~\ref{g1g2potencia}.
\begin{figure}[h]
\includegraphics[scale=0.5]{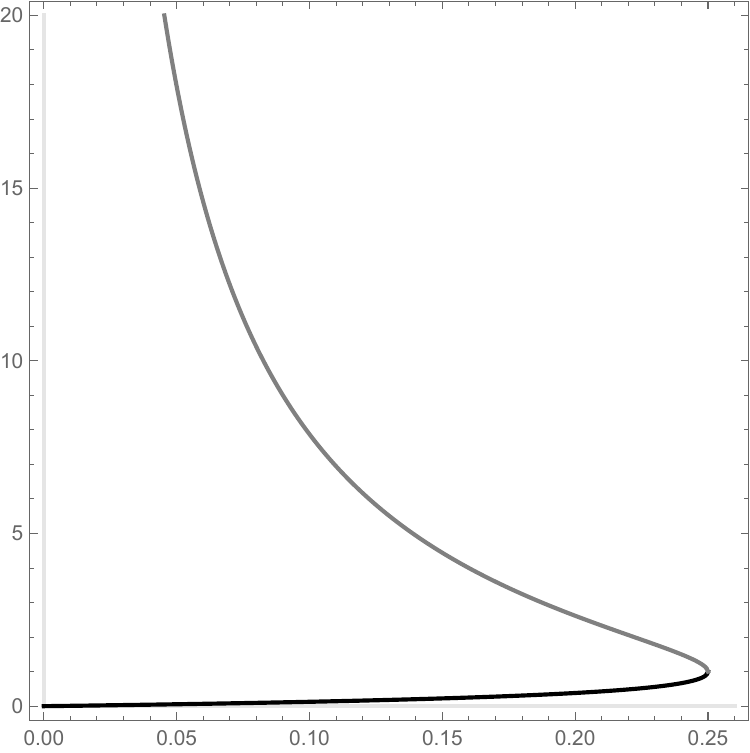}
\caption{$g^{-1}_1(\lambda)$ and $g^{-1}_2(\lambda)$ for $f(s)=(1 +s)^2$.}
\label{g1g2potencia}
\end{figure}

\begin{example}\label{ejempol2}\rm Let $G$ be the graph given in Example \ref{example4pt} and  the same $\Omega$ given there, and consider the Gelfand-type problem $(P(\lambda(1+s)^2))$.  We have that $u$ is a solution to problem $(P(\lambda((1+s)^2))$  if and only if it verifies
$$\left\{\begin{array}{ll}
\displaystyle- \Delta_{m^G} u(2) =\lambda \, (1 +u(2))^2,\\ \\ - \Delta_{m^G} u(3) =\lambda \, (1+u(3))^2,
\\ \\
u(1) = u(4) = 0,
\end{array}\right.
$$
which is equivalent to  $u(1) = u(4) = 0$ and
$$\left\{\begin{array}{ll}
u(2) - \frac12 u(3) = \lambda \, (1 +u(2))^2,
\\ \\
u(3) - \frac12 u(2) = \lambda \, (1 +u(3))^2.
\end{array}\right.
$$
This problem has two symmetric branches of solutions, $u_\lambda$ and $u^\lambda$
  for $0 < \lambda < \lambda^*= \frac18,$
$$u_{\lambda}(2)  = u_{\lambda}(3) =  \frac{4\lambda}{1 - 2\lambda +\sqrt{1- 8 \lambda}},$$
$$    u^{\lambda}(2)  = u_{\lambda}(3) =\frac{ 1 - 4 \lambda  + \sqrt{1- 8 \lambda}}{4\lambda },$$
and, for $\lambda=\lambda^*$,
$u^*(2)  = u^*(3) = 1.$
See Figure~\ref{figplotpol2}.
\begin{figure}[ht]
\includegraphics[scale=0.5]{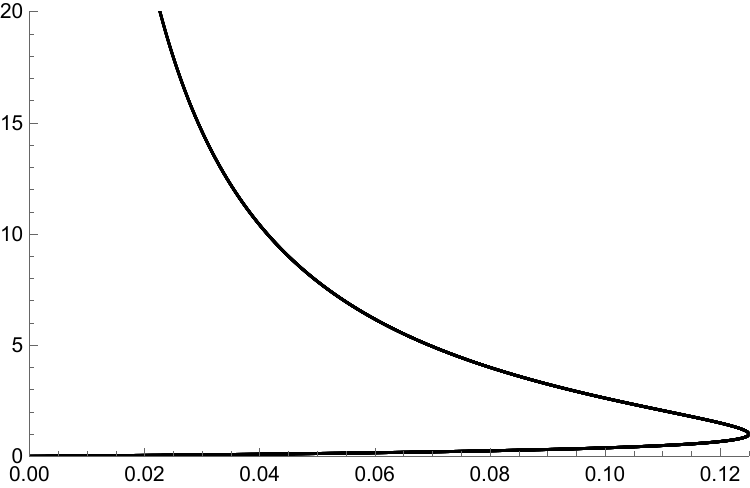}
\caption{Branches of $(\lambda, u)$ in Example~\ref{ejempol2}.}
\label{figplotpol2}
\end{figure}
\hfill $\blacksquare$
\end{example}

  In the next example we consider the same weighted graph but with $f(s)=1+s$, which is  not strictly convex. We see the no existence of extremal solution.

\begin{example}\label{aladm10}\rm
\rm  Let $G$ be the graph given in Example \ref{example4pt} and  the same $\Omega$ given there, and consider the Gelfand-type problem $(P(\lambda(1+s))$.  It is easy to see that
$u$ is a solution of $(P(\lambda(1+s))$  if and only
$$u(2)=u(3)=\frac{\lambda}{\frac12-\lambda}.$$
Hence in this case, we have that
$$\lambda^*=\lambda_m(\Omega)=\frac12,$$ all the solutions are minimal,
there is not extremal solution, see Figure~\ref{dim10}.
\begin{figure}[ht]
\includegraphics[scale=0.4]{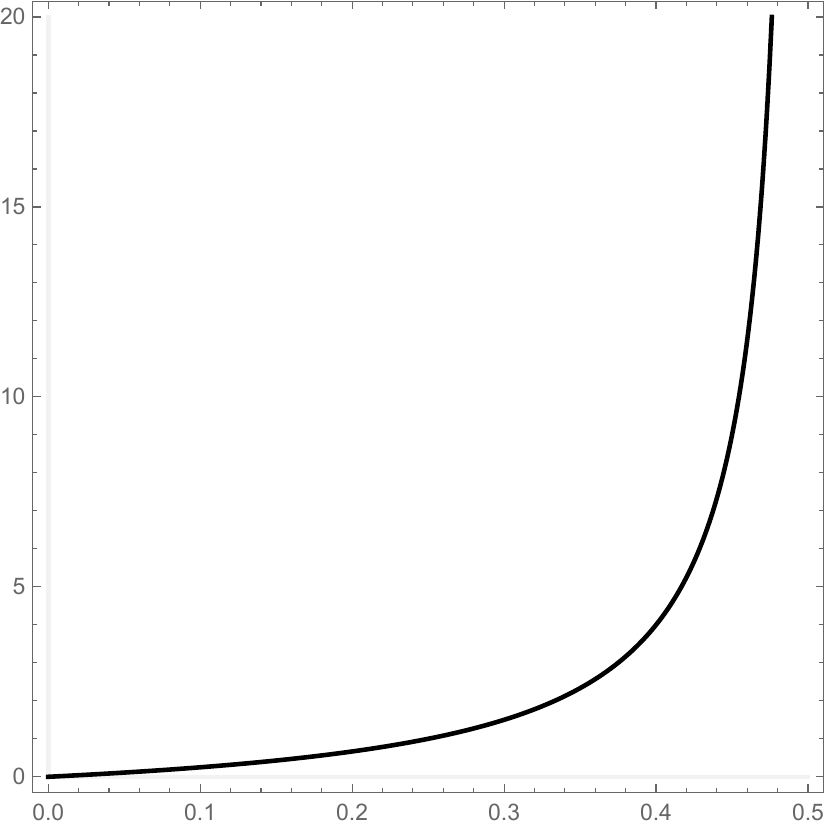}
\caption{Branch  of $(\lambda,u)$ in Example~\ref{aladm10}.}
\label{dim10}
\end{figure}
\hfill $\blacksquare$
\end{example}

\newpage
\section*{Acknowledgements}

The first and third authors  have been partially supported  by  Grant PID2022-136589NB-I00 funded by MCIN/AEI/10.13039/501100011033 and FEDER
and by Grant RED2022-134784-T funded by
MCIN/AEI/10.13039/501100011033.
 The second author is supported by (MCIU) Ministerio de Ciencia, Innovaci\'on y Universidades, Agencia Estatal de Investigaci\'on (AEI) and Fondo Europeo de Desarrollo Regional under Research Project PID2021-122122NB-I00 (FEDER) and by Junta de Andaluc\'ia FQM-116.

\end{document}